\documentclass[12pt]{amsart}

\usepackage{amsmath}
\usepackage{amssymb}
\usepackage{amsthm}

\newcommand\numberthis{\addtocounter{equation}{1}\tag{\theequation}}
\newcommand{\h}{\tilde{h}}
\newcommand{\n}{\hat{n}}
\newcommand{\talpha}{\tilde{\alpha}}
\newcommand{\tbeta}{\tilde{\beta}}
\newcommand{\tgamma}{\tilde{\gamma}}
\newcommand{\tdelta}{\tilde{\delta}}
\newcommand{\teta}{\tilde{\eta}}
\newcommand{\tzeta}{\tilde{\zeta}}

\DeclareMathOperator{\ric}{Ric}

\DeclareMathOperator{\vol}{vol}
\DeclareMathOperator{\tr}{tr}
\DeclareMathOperator{\grad}{grad}

\newtheorem{theorem}{Theorem}[section]
\newtheorem{lemma}[theorem]{Lemma}
\newtheorem{claim}{Claim}

\newtheorem{corollary}[theorem]{Corollary}

\newtheorem{remark}[theorem]{Remark}

\author[Gursky, McKeown, Tyrrell]{Matthew J. Gursky*, Stephen E. McKeown**, and Aaron J. Tyrrell\textsuperscript{\textdagger}}
\address{Department of Mathematics, University of Notre Dame, Notre Dame, IN 46556}
\address{Department of Mathematical Sciences, University of Texas at Dallas, 800 W. Campbell Road, Richardson, TX 75080}
\address{Department of Mathematics and Statistics, Texas Tech University, 1108 Memorial Circle, Lubbock, TX 79409}
\email{mgursky@nd.edu}
\email{stephen.mckeown@utdallas.edu}
\email{aatyrrel@ttu.edu}
\title[Renormalized volume of regions]{Renormalized volume of minimally bounded regions in asymptotically hyperbolic Einstein spaces}

\thanks{*Supported in part by NSF grant DMS-2105460 and DMS-1547292. **Supported in part by Simons Foundation grant 966614. \textdagger Supported in part by the National University of Ireland.}

\begin{document}
\maketitle
\begin{abstract}We define a renormalized volume for a region in an asymptotically hyperbolic Einstein
manifold that is bounded by a Graham-Witten minimal surface and the conformal infinity. We prove
a Gauss-Bonnet theorem for the renormalized volume, and compute its derivative under variations
of the minimal hypersurface.
\end{abstract}
\section{Introduction}

The renormalized volume of an even-dimensional asymptotically hyperbolic Einstein (AHE) manifold $(X^{n + 1},g_+)$ is among its most important global invariants. Introduced in \cite{hs98} (see also \cite{g00}), 
it is defined by taking the
order-zero term in the expansion in $\varepsilon$ of the quantity $\vol_{g_+}(\left\{ r > \varepsilon \right\})$, where $r$ is a so-called geodesic defining function for the boundary at infinity, $M^n$. There are many such
defining functions, and the essential property of the renormalized volume $V_+$ is that it does not depend on which one is chosen. (This is generally not true if $X$ is odd-dimensional.)

One of the basic theorems regarding renormalized volume in dimension four is Anderson's Gauss-Bonnet theorem (\cite{and01}, see also \cite{cqy08}), which states that
\begin{equation}
	\label{agb}
	4\pi^2\chi(X^4) = 3V_+ + \frac{1}{8}\int_X|W_{g_+}|_{g_+}^2dv_{g_+}.
\end{equation}
Here $W_{g_+}$ is the Weyl tensor of $g_+$; since $|W_{g_+}|_{g_+}^2$ is a pointwise conformal invariant of weight $-4$, the integral is guaranteed to converge notwithstanding the infinite volume of $(X,g_+)$.
Anderson used (\ref{agb}) to compute the variation of $V_+$ with respect to changes in $g_+$.

In this paper, we establish analogous results for half of an AHE manifold that has been partitioned into two by a minimal surface. Specifically, suppose $(X^4,g_+)$ is an AHE manifold with conformal infinity
$(M^3,[\bar{h}])$, and suppose further that $Y^3 \subset X$ is a minimal hypersurface that intersects $M$ in a closed manifold $\Sigma^2 = M \cap Y$; we further assume that $Y$ divides $X$ into two
parts, $X^+$ and $X^-$, whose intersection is precisely $Y$ (the assignment of $+$ is arbitrary). Such a setting has been much studied in the literature on AHE manifolds, beginning with \cite{gw99}, which defined the
\emph{renormalized area} of $Y$ in analogy to the renormalized volume of $X$; it has also been and remains a setting of much interest in the physics literature, particularly in the context of the AdS/CFT correspondence.

We will be concerned, not with the renormalized area of $Y$, but with the renormalized \emph{volume} $V_+^+$ of $X^+$, which we may define as the constant term in the expansion
$\vol_{g_+}(\left\{ x \in X^+: r(x) > \varepsilon \right\})$, with $r$ a geodesic defining function. It is not immediately obvious that this quantity is independent of the choice of $r$:
the proof in the global case depends strongly on the product decomposition $[0,\delta)_r \times M$ of a collar neighborhood of $M$ in $X$, but generically there is no such decomposition of a collar neighborhood
of $M^+ = M \cap X^+$ in $X^+$. One could prove using rather more elaborate versions of the arguments of \cite{g00} that $V_+^+$ is invariant in this context, but our interest is in a Gauss-Bonnet formula,
and so we approach the result by a somewhat different path, as described below. 

We note that renormalized volume of regions in AH spaces divided in two by hypersurfaces was considered in \cite{gw19} using quite different techniques. The authors showed that a volume could be defined
in quite general circumstances -- in particular, not assuming the Einstein or minimality conditions -- but did not show that it is well-defined independent of all choices in the four-dimensional Einstein case.

Let $N \subset \mathring{X}$ be any hypersurface, and let $h = g_+|_{TN}$ be the 
induced metric on $N$.
Define an extrinsic curvature quantity $\mathcal{C}_N$ on $N$ by the formula
\begin{equation*}
	\mathcal{C}_N = \mathring{L}_N^{\alpha\beta}R_{\alpha\beta}^{g_+}
	- \mathring{L}^{\alpha\beta}_NR_{\alpha\beta}^{h} 
	+ \frac{1}{3}H_N|\mathring{L}_N|_{h}^2 -
	\frac{1}{3}\tr_{h}\mathring{L}_N^3.
\end{equation*}
Here $L_N$ is the second fundamental form of $N$ and $\mathring{L}_N$ its
tracefree part, while $H_N = h^{\alpha\beta}L_{\alpha\beta}$ is its mean curvature.
The curvature terms appearing are the Ricci tensors of the respective metrics, and
$\alpha, \beta$ are indices on $TN$. It is easy to show (and will be shown within) that
$\mathcal{C}_N$ is a pointwise conformal invariant of weight $-3$; indeed, in the notation of \cite{cq97i},
$\mathcal{C}_N = -\frac{1}{2}\mathcal{L}_4 - \frac{1}{3}\mathcal{L}_5$.

The first main result of this paper is the following.
\begin{theorem}
	\label{gbthm}
	Let $(X^4,g_+)$ be an asymptotically hyperbolic space satisfying the
	Einstein condition $\ric(g_+) = -3g_+$, with conformal infinity
	$(M^3,[\bar{h}])$. Let $Y^3$ be a complete minimal hypersurface dividing 
	$X$ into two pieces $X^+$ and $X^-$ such that $X^+ \cap X^- = Y$ and
	such that $Y \cap M = \Sigma^2 \neq \emptyset$. Let $r$ be a fixed
	geodesic defining function for $M$, and let $V_+^+$ be the constant
	term in the expansion
	\begin{equation*}
		\vol_{g_+}\left( \left\{ x \in X^+:r(x) > \varepsilon \right\} \right) = c_0\varepsilon^{-3} + c_2\varepsilon^{-1} + V_+^+ + o(1).
	\end{equation*}
	Let $\tilde{h} = g_+|_{TY}$.
	Then
	\begin{equation}
		\label{gbrv}
		\pi^2(4\chi(X^+) - \chi(\Sigma^2)) = 3V_+^+ + \frac{1}{8}\int_{\mathring{X}^+}|W_{g_+}|_{g_+}^2dv_{g_+} + \int_{\mathring{Y}}\mathcal{C}_{Y}dv_{\tilde{h}}.
	\end{equation}
\end{theorem}

One then immediately obtains

\begin{corollary}
	The renormalized volume $V_+^+$ is independent of the choice of geodesic defining function $r$, and it satisfies (\ref{gbrv}).
\end{corollary}

A natural question about the newly defined renormalized volume is how it changes if $Y$ is varied through minimal surfaces in $X$. The second main result of the paper is as follows.
\begin{theorem}
	\label{varthm}
	Let $X, M, Y, \Sigma, X^+, g_+$, $\bar{h}$, and $V_+^+$ be as in Theorem \ref{gbthm}. Suppose that $\mathcal{F}:(-\varepsilon,\varepsilon)_t \times Y \to X$ is a $C^3$ variation of $Y$ through
	minimal surfaces in $X$, so that $\mathcal{F}(t,\Sigma) \subset M$ for all $t$. Let $\widetilde{\mathcal{F}} = \mathcal{F}|_{(-\varepsilon,\varepsilon) \times \Sigma}$.
	Define $\tilde{f} \in C^{\infty}(\Sigma)$ by $\tilde{f} = \left\langle \left.\frac{d}{dt}\right|_{t = 0}\widetilde{\mathcal{F}}, \bar{\nu}_M\right\rangle$, where $\bar{\nu}_M$ is the inward-pointing
	normal vector to $\Sigma$ in $M^+$ with respect to $\bar{h}$. Define $f \in C^{\infty}(\mathring{Y})$ by $f = \left\langle \left.\frac{d}{dt}\right|_{t = 0}\mathcal{F},\mu_Y\right\rangle_{g_+}$,
	where $\mu_Y$ is the $(X^+,g_+)$-inward unit normal vector along $Y$.
	Let $r$ be a geodesic defining function near $M$. Then 
	\begin{equation*}
		\left.\frac{d}{dt}\right|_{t = 0}V_+^+ = \frac{1}{2} \oint_{\Sigma}\tilde{f} g^{(3)}(\bar{\nu}_M,\bar{\nu}_M) dv_{\bar{k}} +  \frac{1}{3} f.p. \int_{\mathring{Y}} f |\mathring{L}_Y|_{\tilde{h}}^2dv_{\tilde{h}},
	\end{equation*}
	where $\bar{k} = \bar{h}|_{T\Sigma}$, $\tilde{h} = g_+|_{TY}$, $g^{(3)}$ is the nonlocal term in the expansion in $r$ of $g_+$, $\n$ is the index corresponding to $\mu_Y$,
	and $f.p. \int_{\mathring{Y}} f |\mathring{L}_Y|_{\tilde{h}}^2dv_{\tilde{h}}$ denotes the zeroth-order part, in $\varepsilon$, of
	$\int_{Y \cap \left\{ r > \varepsilon \right\}}f|\mathring{L}_Y|_{\tilde{h}}^2dv_{\tilde{h}}$.
\end{theorem}
For more about the nonlocal term $g^{(3)}$, see (\ref{hbarexp}) and the surrounding discussion. We show in Lemma \ref{finlem} that the finite part of the integral over $Y$ can be written as the convergent integral of a rather more complicated expression.

The above theorem is stated for variations of $Y$ through minimal surfaces, whose existence in general we do not assert.  However, one can broaden the definition of $V_+^+$ to any dividing hypersurface by 
using (\ref{gbrv}). In that case, Theorem \ref{varthm} remains valid for any variation of $Y$ that preserves minimality to first order; see section \ref{varsec}, where we also explain why $C^3$-regularity of 
such a variation is in general optimal.  

In considering the existence problem for the variation of $Y$, the required boundary data would be the induced variation of $\Sigma$, so
another natural question is whether the derivative $\dot{V}_+^+$ only depends on the induced normal variation $\tilde{f}$.  For example, suppose there are two variations of $Y$ through minimal surfaces that induce 
the same variation of $\Sigma$; do the derivatives of $V_{+}^{+}$ with respect to these variations agree? The answer is yes, at least if $|\mathring{L}_Y|_{\h}^2 \leq 3$ everywhere; see Lemma \ref{fLemma}.

These theorems may be interpreted physically within the AdS/CFT correspondence of high-energy and condensed matter physics. To do so, we assume that $(M^3,[\bar{h}])$ is a spacelike slice
within a static four-dimensional conformal field theory $\Omega$; and that $(X^4,g_+)$ is an Einstein spacelike slice within a static asymptotically anti-de-Sitter Einstein five-dimensional spacetime $Z$
with conformal infinity $\Omega$.
The surface $\Sigma$ is then known as an entangling surface between
$M^+$ and $M^-$, and $Y$ is the so-called Ryu-Takayanagi surface extending $\Sigma$. According to the ``volume = complexity'' conjecture (\cite{s16,bc16,cfn17,abn18,jkkt20}), then, $V_+^+$ encodes the algorithmic complexity
of the quantum state of $M^+$. The above theorems can then be interpreted as giving formulae for this complexity and for its derivative as the entangling surface $\Sigma$ is varied continuously, so long
as $Y$ also varies continuously. (As demonstrated in \cite{bc16}, the latter will not always be the case.)

The assumption that $X$ and its five-dimensional ambient Lorentzian manifold $Z$ are both Einstein, of course, is rather restrictive. In general physical situations, one might expect that the
Ricci tensor of $X$ includes some extrinsic terms. But even if so, these would have well-defined asymptotics due to the asymptotically AdS condition on $Z$, and it would be straightforward, if tedious,
to carry out our calculation the same way in that context.

In section \ref{notsec}, we introduce our setting and notation. In section \ref{gbsec}, we prove Theorem \ref{gbthm}; and in section \ref{varsec}, we prove Theorem \ref{varthm}.

\section{Setting and Notation}
\label{notsec}

Recall that an asymptotically hyperbolic (AH) manifold is a compact manifold $X^{n + 1}$ with boundary $M^n$, equipped on the interior $\mathring{X}$ with a metric $g_+$ such that, for any defining function
$\varphi$ for $M$, the metric $\bar{g} = \varphi^2g_+$ extends to a Riemannian metric on $X = \overline{X}$; and such that, in addition,
$|d\varphi|_{\bar{g}} = 1$ along $M$.  The optimal regularity of $\bar{g}$ is in general a delicate question, but in the context of this paper (i.e., $X$ is four dimensional) by a result of Chru\'{s}ciel-
Delay-Lee-Skinner \cite{cdls05} we may assume that there is a compactification such that $\bar{g}$ is smooth up to the boundary.  The canonical example of an AH manifold is hyperbolic space itself, where $X$ is the unit ball $\mathbb{B}^{n + 1}$, and the metric is $g_H = \frac{4|dx|^2}{(1 - |x|^2)^2}$. Given an AH metric, the metric $\bar{h} = \bar{g}|_{TM}$ is a metric on $M$, but is not well defined since the choice of $\varphi$ is arbitrary. However, the conformal
class $[\bar{h}]$ is well defined, and is called the \emph{conformal infinity}.

A defining function $r$ for $M$ is called geodesic if $|dr|_{r^2g_+} = 1$ on a neighborhood of $M$. Such a function induces a diffeomorphism 
\begin{equation}
	\label{normformid}
	\psi:[0,\varepsilon)_r \times M \hookrightarrow X
\end{equation}
onto a neighborhood of $M$ in $X$ such that 
\begin{equation}
	\label{gnf}
	\psi^*g_+ = \frac{dr^2 + \bar{h}_r}{r^2},
\end{equation}
where $\bar{h}_r$ is a one-parameter family of metrics on $M$. A lemma of Graham-Lee (\cite{gl91}) states that geodesic defining functions are in one-to-one
correspondence with the representatives $\bar{h}$ of $[\bar{h}]$, according to the correspondence $\bar{h}_0 = \bar{h}$. The form (\ref{gnf}) is called the geodesic normal form corresponding to
$\bar{h} = \bar{h}_0$. We may assume that any geodesic compactification of $X$ is smooth (\cite{cdls05}).

An AH metric is called Einstein (or AHE) if it satisfies as well the condition $\ric(g) + ng = 0$. We will be concerned exclusively with four-dimensional AHE spaces, i.e. the case $n = 3$.
In this case, it is known (\cite{fg85,fg12,g00}) that in geodesic normal form, $\bar{h}_r$ has the expansion
\begin{equation}
	\label{hbarexp}
	\bar{h}_r = \bar{h} - r^2P^{\bar{h}} + r^3g^{(3)} + O(r^4),
\end{equation}
where $\tr_{\bar{h}}g^{(3)} = 0$ and where $P^{\bar{h}}$ is the Schouten tensor of $\bar{h}$, given by
\begin{equation}
	\label{scheq}
	P_{\mu\nu}^{\bar{h}} = R^{\bar{h}}_{\mu\nu} - \frac{1}{4}R_{\bar{h}}\bar{h}_{\mu\nu}.
\end{equation}
Apart from the trace condition, the tensor $g^{(3)}$ is not locally determined by the geometry of $(M^3,\bar{h})$.

The renormalized volume of $(X,g_+)$ is defined as follows (\cite{hs98,g00}). Choose a metric $\bar{h} \in [\bar{h}]$, and let $r$ be the corresponding geodesic defining function. Then the
set $\left\{ r > \varepsilon \right\}$ has volume
\begin{equation}
	\label{rnexpeq}
	\vol_{g_+}(\left\{ r > \varepsilon \right\}) = c_0\varepsilon^{-3} + c_2\varepsilon^{-1} + V_+ + o(1).
\end{equation}
The renormalized volume is $V_+$, and it is independent of the choice of $\bar{h}$ (that is, of $r$).

In our setting of interest, there exists as well an orientable minimal surface $Y^3 \subset X$, intersecting $M$ transversely in a closed two-manifold $\Sigma^2 = Y \cap M$, and dividing $X$ into two connected pieces
$X^+$ and $X^-$ such that $Y = X^+ \cap X^-$. 
We write $M^+ = X^+ \cap M$ and $M^- = X^- \cap M$, so that $\Sigma = M^+ \cap M^-$. The assignment of the signs $+$ and $-$ is arbitrary, and corresponds to a choice of unit normal vector field on $Y$.

We now introduce the notations we will use. We let $(X^4,M^3,g_+)$ be an AHE space, and $Y^3 \subset X$ a minimal surface as above. We will let $[\bar{h}]$ be the conformal infinity, and corresponding to the metric
$\bar{h}$ will be the geodesic defining function $r$. The compactified metric is $\bar{g} = r^2g_+$. 
Furthermore, $X^+, M^+$, and $\Sigma^2$ will be as above. For $\varepsilon > 0$, we let $X_{\varepsilon} = \left\{ r > \varepsilon \right\}$,
with $X_{\varepsilon}^+ = X^+ \cap X_{\varepsilon}$. We set $Y_{\varepsilon} = \overline{Y \cap X_{\varepsilon}}$ and $M_{\varepsilon} = \left\{ r = \varepsilon \right\}$. Similarly we set
$M_{\varepsilon}^+ = X^+ \cap M_{\varepsilon}$. Finally, $\Sigma_{\varepsilon} = Y \cap M_{\varepsilon}^+$.

Next, there are a number of metrics to name. We let $h_{\varepsilon} = g_+|_{TM_{\varepsilon}}$, while $\bar{h}_{\varepsilon} = \varepsilon^2h_{\varepsilon} = \bar{g}|_{TM_{\varepsilon}}$.
We let $\tilde{h} = g_+|_{TY}$, while $\bar{\tilde{h}} = r^2\tilde{h} = \bar{g}|_{TY}$. We let $\bar{k} = \bar{g}|_{T\Sigma}$, while $k_{\varepsilon} = g_+|_{T\Sigma_{\varepsilon}}$ and
$\bar{k}_{\varepsilon} = r^2k_{\varepsilon} = \varepsilon^2k_{\varepsilon}$. The decorations of $\varepsilon$ will sometimes change position as needed; for example, we will write
$h_{\mu\nu}^{\varepsilon}$, but $h_{\varepsilon}^{\mu\nu}$.

Now, near $\Sigma \subset M$, we can uniquely solve the eikonal equation and find $w \in C^{\infty}(M)$ such that $|dw|_{\bar{h}}^2 = 1$ near $\Sigma$, $w|_{\Sigma} = 0$, 
and $w \geq 0$ on $M^+$. The metric $\bar{h}$ then takes the form
$\bar{h} = dw^2 + \bar{k}_{w}$, with $\bar{k}_{w}$ a one-parameter family of metrics on $\Sigma$. Near any point $p \in \Sigma$, we can choose coordinates $x^1,x^2$ on a neighborhood of $p$ in
$\Sigma$; then by the flow of $\grad_{\bar{h}}w$ on $M^+$, the system $(x^1,x^2,x^3 = w)$ extends to a coordinate system on a neighborhood of $p$ in $M$. Finally, by the flow
of $\grad_{\bar{g}}r$, the system $(r = x^0,x^1,x^2,x^3 = w)$ extends to a coordinate system on a neighborhood of $p$ in $X$. Now, we will regard $Y$ as given by a function
\begin{equation}
	\label{wueq}
	w = u(r,x^1,x^2),
\end{equation}
where $u(0,x^1,x^2) \equiv 0$. This is the same convention as in \cite{gw99}. In fact, we may regard a neighborhood of $\Sigma$ in this way as a product
$[0,\varepsilon)_r \times \Sigma \times (-\varepsilon,\varepsilon)_{w}$; when using this product identification, we will use $\zeta$ to refer to a point of $\Sigma$, so that a generic point may be written $(r,\zeta,w)$.

When using index notation locally, we will let $0 \leq i,j \leq 3$ be indices on $TX$; $1 \leq \mu,\nu \leq 3$ be indices on $TM$; and $1 \leq a,b \leq 2$ be indices on $T\Sigma$. We also let
$0 \leq \alpha, \beta \leq 2$, which we will use when discussing $TY$.

Turning to extrinsic geometry, we let $\bar{\mu}_{M_{\varepsilon}}, \bar{\mu}_{Y}$ be the $X^+$-inward unit $\bar{g}$-normal to the given hypersurface; the unbarred versions will refer to
the unit normal with respect to $g_+$. We let $\bar{\nu}_{M_{\varepsilon}}$ be the $\bar{g}$-unit
normal to $\Sigma_{\varepsilon}$ that is directed into $M^+$, and $\bar{\nu}_{Y_{\varepsilon}}$, similarly, the $Y_{\varepsilon}$-inward $\bar{g}$-unit normal to $\Sigma_{\varepsilon}$. We let
$\overline{L}_{M_{\varepsilon}}, \overline{L}_{Y}$ be the second fundamental forms of the indicated hypersurfaces with respect to the inward unit normals $\bar{\mu}_{M_{\varepsilon}}$ and $\bar{\mu}_{Y}$, and computed
with respect to $\bar{g}$. Thus, for example,
\begin{equation*}
	\overline{L}_Y(A,B) = -\langle \nabla^{\bar{g}}_A\bar{\mu}_Y,B\rangle.
\end{equation*}
The tracefree parts are denoted $\mathring{\overline{L}}_{M_{\varepsilon}}$, etc. In all of these, we will sometimes write the hypersurface in the upper position, should it be convenient to do
so to place covariant indices; similarly, an unbarred $L$ will refer to the second fundamental form with respect to $g_+$ instead of $\bar{g}$. We let $\overline{H}_{M_{\varepsilon}} =
\bar{h}^{\mu\nu}_{\varepsilon}\overline{L}_{\mu\nu}^{M_{\varepsilon}}$ be the mean curvature of $M_{\varepsilon}$ with respect to $\bar{g}$ (or, if we omit the $\varepsilon$, that of $M$); similarly
for $\overline{H}_Y$, while $H_{M_{\varepsilon}}$ and $H_Y$ are the same quantities with respect to $g_+$ (recall we assume $H_Y \equiv 0$).  We let $\overline{II}_{Y_{\varepsilon}}$ be the second fundamental form of $\Sigma_{\varepsilon}$
viewed as a hypersurface of $Y_{\varepsilon}$ with respect to $\bar{\tilde{h}}$, while $\overline{II}_{M_{\varepsilon}}$ is the same for $\Sigma_{\varepsilon}$ viewed as a hypersurface in $M_{\varepsilon}$
with respect to $\bar{h}_{\varepsilon}$. The traces of these (i.e., the mean curvatures of $\Sigma_{\varepsilon}$ viewed as a hypersurface of the respective three-manifold) we denote $\bar{\eta}_{Y_{\varepsilon}},
\bar{\eta}_{M_{\varepsilon}}$. Again, the unbarred versions are with respect to the unbarred metrics $\tilde{h}$ and $h_{\varepsilon}$. We also let $\bar{\eta}_M$ be the mean curvature of $(\Sigma,\bar{k}) \subset
(M,\bar{h})$.

We define a smooth function $\theta_0^{\varepsilon} \in C^{\infty}(\Sigma_{\varepsilon})$ to be the angle, at each point, between $Y$ and $M_{\varepsilon}$; that is,
$\cos(\theta_0^{\varepsilon}) = -\langle\bar{\mu}_Y,\bar{\mu}_{M_{\varepsilon}}\rangle$. If the $\varepsilon$ is omitted, then it denotes the angle between $M$ and $Y$ at a point of $\Sigma$. Since
$\theta_0^{\varepsilon}$ is manifestly a conformal invariant, we do not distinguish between barred and unbarred versions.

Our curvature convention is such that the Ricci tensor is given by $R_{ij} = R^k{}_{ikj}$.

If $A$ is a vector or tensor field, we write $A = O_{\bar{g}}(\varphi)$, for $\varphi$ a function, whenever
$|A|_{\bar{g}} = O(\varphi)$.

\section{The Gauss-Bonnet Formula}
\label{gbsec}

We now prove Theorem \ref{gbthm}. We do so by using a form of the Gauss-Bonnet formula that has good conformal invariance properties, which allows us to compute using $\bar{g}$ instead of $g_+$.

\begin{proof}[Proof of Theorem \ref{gbthm}]
	Let $(X,M,g_+)$ be an AHE space with conformal infinity $[\bar{h}]$, and let $Y$ be as in the previous section. Let $\bar{h} \in [\bar{h}]$, and let $r$ be the corresponding geodesic defining function. Let
	$\varepsilon > 0$.
	Then $\overline{X_{\varepsilon}^+}$ is a four-manifold with codimension-two corner $\Sigma_{\varepsilon}$, and boundary hypersurfaces $M_{\varepsilon}^+$ and $Y_{\varepsilon}$ (see section \ref{notsec} for all notation).
	The Gauss-Bonnet theorem for Riemannian manifolds with corners (in this case $X_{\varepsilon}^+$), proven first in \cite{aw43} (and see \cite{c45}), can be rewritten in the following 
	conformally useful way (\cite{m20}, building on \cite{cq97i}).
	\begin{equation}
		\label{gba}
		\begin{split}
			4\pi^2\chi(X_{\varepsilon}^+) =& \int_{X_{\varepsilon}^+}\left( \frac{1}{8}|W_{g_+}|_{g_+}^2 + \frac{1}{2}Q_{g_+} \right)dv_{g_+}
			+ \int_{Y_{\varepsilon}}\left( \mathcal{L}_{Y} + T_Y \right)dv_{\tilde{h}}\\
			&+ \int_{M_{\varepsilon}^+}\left( \mathcal{L}_{M_{\varepsilon}} + T_{M_{\varepsilon}} \right)dv_h + \oint_{\Sigma_{\varepsilon}}\left( U_{\Sigma_{\varepsilon}} + G_{\Sigma_{\varepsilon}} \right)
			dv_{k_{\varepsilon}}.
		\end{split}
	\end{equation}
	Here, $W_{g_+}$ is the Weyl tensor of $g_+$, and the norm in question is its two-tensor norm $W^{ijkl}W_{ijkl}$. Meanwhile, $Q_{g_+}$ is the $Q$-curvature of $g_+$, defined for any metric $g$ by
	\begin{equation*}
		Q_g = -\frac{1}{6}\Delta_gR_g + \frac{1}{6}R_g^2 - \frac{1}{2}R^{ij}_gR_{ij}^g.
	\end{equation*}
	Here, the Laplacian is a negative operator and the curvatures are respectively the scalar and Ricci curvatures of $g$. For any metric $g$, the quantity $|W_{g}|_{g}^2dv_{g}$ is a pointwise conformal invariant
	of weight zero. Under a conformal transformation $\tilde{g} = e^{2\omega}g$, the $Q$ curvature transforms as
	\begin{equation*}
		e^{4\omega}Q_{\tilde{g}} = Q_g + P_4^g\omega,
	\end{equation*}
	where $P_4^g$ is the \emph{Paneitz operator} associated to $g$; we will not use the Paneitz operator and so omit it here.

	We give the definition of $\mathcal{L}_N$ and $T_N$, due to \cite{cq97i}, for an arbitrary boundary hypersurface $(N^3,h)$ embedded in a four-manifold endowed with metric $g$. The definition is
	\begin{equation}
		\label{Leq}
		\mathcal{L}_{N} = \mathring{L}_{N}^{\mu\nu}R_{\mu\nu}^{g} - 2\mathring{L}_{N}^{\mu\nu}R_{\mu\nu}^{h} + 
		\frac{2}{3}H_{N}|\mathring{L}_{N}|_{h}^2 - \tr_{h}\mathring{L}_{N}^3,
	\end{equation}
	where $L_N$ and $H_N$ are the second fundamental form and the mean curvature as before, and $\mu,\nu$ are indices on $TN$. Similarly, the $T$-curvature is defined by
	\begin{equation}
		\label{Teq}
		\begin{split}
			T_N =& -\frac{1}{12}\mu(R_g) - \mathring{L}_N^{\mu\nu}R_{\mu\nu}^g + \mathring{L}^{\mu\nu}_NR^h_{\mu\nu}-\frac{1}{2}H_N|\mathring{L}_N|_{h}^2 + \frac{2}{3}\tr_h\mathring{L}_N^3\\
			&+ \frac{1}{6}R_hH_N - \frac{1}{27}H^3_N - \frac{1}{3}\Delta_hH_N,
		\end{split}
	\end{equation}
	where $\mu$ is the inward-pointing unit normal to $N$.
	Under the conformal change $\tilde{g} = e^{2\omega}g$, this transforms according to the equation
	\begin{equation}
		\label{Transeq}
		e^{3\omega}\widetilde{T}_N = T_N + P_3^g\omega,
	\end{equation}
	where $P_3^g:C^{\infty}(X) \to C^{\infty}(N)$ is the conformally covariant boundary operator
	\begin{equation}
		\label{P3eq}
		\begin{split}
			P_3^gf =& \frac{1}{2}\mu\Delta_gf + \Delta_{h}\mu(f) - \frac{1}{3}H_N\Delta_hf + \mathring{L}_N^{\mu\nu}\nabla_{\mu}^h\nabla_{\nu}^hf + \frac{1}{3}H_N^{\mu}f_{\mu}\\
			&+ \left( \frac{1}{6}R_g - \frac{1}{2}R_h - \frac{1}{2}|\mathring{L}_N|_{h}^2 + \frac{1}{3}H_N^2 \right)\mu(f).
		\end{split}
	\end{equation}
 (We note that this formula differs from that in \cite{m20}; that paper and others in the
 literature contain misprints in the formula, which we have corrected by
 \cite{cq97i}.)

	Next we turn to the corner quantities. For a corner $(\Xi,k)$ that forms the intersection between two boundary hypersurfaces $N$ and $S$ making angle $\theta_0 \in C^{\infty}(\Xi)$, $G$ is defined by 
	\begin{equation}
		\label{Geq}
		G_{\Xi} = \frac{1}{2}\cot(\theta_0)(|\mathring{II}_N|_k^2 + |\mathring{II}_S|_k^2) - \csc(\theta_0)\mathring{II}_{ab}^N\mathring{II}_S^{ab},
	\end{equation}
	where $II$, etc., are as in section \ref{notsec}. The $G$ curvature is a pointwise conformal invariant of weight $-2$ (when the ambient metric on the four-manifold is changed conformally).
	Next, $U_{\Xi}$ is defined by
	\begin{equation}
		\label{Ueq}
		U_{\Xi} = (\pi - \theta_0)K_{\Xi} - \frac{1}{4}\cot(\theta_0)(\eta_N^2 + \eta_S^2) + \frac{1}{2}\csc(\theta_0)\eta_N\eta_S - \frac{1}{3}(\nu_NH_N + \nu_SH_S).
	\end{equation}
	Here, $K_{\Xi}$ is the Gaussian curvature of $\Xi$, and the other quantities are defined analogously to those in the previous section. Under a global conformal change $\tilde{g} = e^{2\omega}g$,
	$U$ transforms according to the equation
	\begin{equation}
		\label{Utranseq}
		e^{2\omega}\widetilde{U}_{\Xi} = U_{\Xi} + P_2^g\omega,
	\end{equation}
	where $P_2^g:C^{\infty}(X) \to C^{\infty}(\Xi)$ is the conformally covariant operator
	\begin{equation}
		\label{P2eq}
		\begin{split}
			P_2^gf =& (\theta_0 - \pi)\Delta_kf + \nu_N\mu_Nf + \nu_S\mu_Sf\\
			&+ \cot(\theta_0)(\eta_N\nu_Nf + \eta_S\nu_Sf) - \csc(\theta_0)(\eta_S\nu_Nf + \eta_N\nu_Sf)\\
			&+ \frac{1}{3}(H_N\nu_Nf + H_S\nu_Sf).
		\end{split}
	\end{equation}

	We now analyze formula (\ref{gba}) in the context of our space $(X_{\varepsilon}^+,g_+)$. Because $|W_{g_+}|_{g_+}^{2}dv_{g_+}$ is a pointwise conformal invariant of weight zero,
	its integral converges as $\varepsilon \to 0$ to $\int_{X^+} |W_{\bar{g}}|_{\bar{g}}^2dv_{\bar{g}}$, which in particular is finite.

	In our setting, $R_{ij}^{g_+} = -3g_{ij}^+$ and $R_{g_+} \equiv -12$, so $\Delta_{g_+}R_{g_+} \equiv 0$ and $Q_{g_+} \equiv 6$. The integral of $\frac{1}{2}Q_{g_+}$ therefore is simply the integral of $3$,
	so the second integral over $X_+$ becomes simply $3\vol_{g_+}(\left\{ r > \varepsilon \right\} \cap X^+)$, which is the same quantity considered in (\ref{rnexpeq}), except that the latter
	is over all of $X$ instead of $X^+$. To compute the contribution from this integral, we consider four different regions of $X$. First, let $r_0 > 0$ be small -- sufficiently small, in particular,
	that the geodesic normal form (\ref{gnf}) holds for $r < 2r_0$, and that the region $\mathcal{U} = \{r < 2r_0, -2r_0 < w < 2r_0\}$ has the decomposition $[0,2r_0) \times \Sigma \times (-2r_0,2r_0)$, with
	$|u(r,\zeta)| < \frac{1}{2}r_0$ on $\mathcal{U}$. Having chosen $r_0$, we will leave it fixed for all time.

	The first region of interest to us is then $A = \left\{ p \in X^+: r(p) \geq r_0 \right\}$. (This set does not depend on $\varepsilon$, which we assume is smaller than
	$r_0$.) Next, we want to capture the points near the boundary $M_{\varepsilon}^+$. The obvious set to consider is $B_{\varepsilon} = (\varepsilon,r_0) \times M^+$. The problem is that this may
	omit points that are contained in $X^+$ or include points contained in $X^-$, because $Y$ is given not by $w = 0$ but by $w = u(r,\zeta)$, where $u$ may be positive or negative away
from $\{0\}\times\Sigma$. To address this, we need to add the volume of the omitted points, $C_{\varepsilon}$, and subtract the volume of the over-included points $D_{\varepsilon}$, viz.,
	\begin{equation*} 
	X^{+}_{\epsilon} = (A \cup  B_{\epsilon} \cup C_{\epsilon} ) \setminus D_{\epsilon}.
	\end{equation*}

	To proceed, we analyze the volume form $dv_{g_+}$. First, at all points, we have $dv_{g_+} = r^{-4}dv_{\bar{g}}$. Near $M$, we can write
	\begin{equation*}
		dv_{\bar{g}} = dv_{\bar{h}_r}dr
	\end{equation*}
	using the normal-form identification (\ref{normformid}). Now in local coordinates $(r,x^1,x^2,x^3)$ near $M$, we may write
	\begin{equation*}
		dv_{\bar{h}_{r}} = \sqrt{\frac{\det(\bar{h}_r)}{\det(\bar{h})}}dv_{\bar{h}}.
	\end{equation*}
	As shown for example in \cite{g00}, we have the expansion
	\begin{equation*}
		\sqrt{\frac{\det(\bar{h}_r)}{\det(\bar{h})}} = 1 + v^{(2)}r^2 + v^{(4)}r^4 + O(r^5),
	\end{equation*}
	where $v^{(2)}, v^{(4)} \in C^{\infty}(M)$ are the so-called \emph{renormalized volume coefficients.} Either by direct computation using (\ref{hbarexp}) or by using
	equation (4.5) and the equation at the top of the same page of (\cite{g17}) (remembering that $M$ is totally geodesic with respect to $\bar{g}$ and that the singular Yamabe metric
	for $\bar{g}$ is $g_+$), we may show that $v^{(2)} = -\frac{1}{8}R_{\bar{h}}$. Thus,
	\begin{align*}
		dv_{g_+} &= r^{-4}\left( 1 - \frac{1}{8}r^2R_{\bar{h}} + O(r^4)\right)dv_{\bar{h}}dr\\
		&= \left( r^{-4} - \frac{1}{8}r^{-2}R_{\bar{h}} + O(1) \right)dv_{\bar{h}}dr.
	\end{align*}

	We next derive an expression for $dv_{\bar{g}}$ (and thus $dv_{g_+}$) near $\Sigma$. Since $\bar{h} = dw^2 + \bar{k}_{w}$ near $\Sigma$, we have
	\begin{align*}
		dv_{\bar{h}} &= \sqrt{\frac{\det(\bar{k}_w)}{\det(\bar{k})}}dv_{\bar{k}}dw\\
		&= (1 + O(w))dv_{\bar{k}}dw.
	\end{align*}
	Hence, near $\Sigma$, we have
	\begin{equation*}
		dv_{g_+} = \left( r^{-4} - \frac{1}{8}r^{-2}R_{\bar{h}} + O(1) \right)(1 + O(w))dv_{\bar{k}}dwdr.
	\end{equation*}

	We then have
	\begin{align*}
		\vol_{g_+}(X_{\varepsilon}^+) &= \vol_{g_+}(A) + \vol_{g_+}(B_{\varepsilon}) + \vol_{g_+}(C_{\varepsilon}) - \vol_{g_+}(D_{\varepsilon})\\
		&= \vol_{g_+}(A) + \int_{M^+}\int_{\varepsilon}^{r_0}\left( r^{-4} - \frac{1}{8}r^{-2}R_{\bar{h}} + O(1) \right)drdv_{\bar{h}}\\
		&\quad - \oint_{\Sigma}\int_{\varepsilon}^{r_0}\int_0^{u(r,\zeta)}\left( r^{-4} + O(r^{-2}) \right)(1 + O(w))dwdrdv_{\bar{k}}(\zeta)\numberthis\label{cornervolform},
	\end{align*}
	where the last integral represents $\vol_{g_+}(C_{\varepsilon}) - \vol_{g_+}(D_{\varepsilon})$.
	Now, by equations (2.13) and (2.14) in \cite{gw99},
	\begin{equation}
		\label{uexp}
		u(r,\zeta) = \frac{1}{4}r^2\overline{\eta}_M(\zeta) + r^{4}\log(r)v(\zeta) + O(r^4),
	\end{equation}
	where $\overline{\eta}_M$ is the mean curvature of $\Sigma$ viewed as a hypersurface of $(M,\bar{h})$ and $v \in C^{\infty}(\Sigma)$. Thus, we find
	\begin{align*}
		3\vol_{g_+}(X_{\varepsilon}^+) &= 3\vol_{g_+}(A) + 3\int_M\int_{\varepsilon}^{r_0}\left( r^{-4} - \frac{1}{8}r^{-2}R_{\bar{h}} + O(1) \right)dr dv_{\bar{h}}\\
		&\quad -3\oint_{\Sigma}\int_{\varepsilon}^{r_0}\left( \frac{1}{4}r^{-2}\overline{\eta}_M + v\log(r) + O(1) \right)drdv_{\bar{k}}\\
		&= \varepsilon^{-3}\vol_{\bar{h}}(M^+) - \varepsilon^{-1}\left(\frac{3}{8}\int_{M^+}R_{\bar{h}}dv_{\bar{h}} + \frac{3}{4}\oint_{\Sigma}\overline{\eta}_Mdv_{\bar{k}}\right)\\
		&\quad + 3V_+^+\numberthis\label{Xint} + o(1).
	\end{align*}
	Here $V_+^+$ is the collection of all the order-zero terms in $\varepsilon$ in the volume expansion, and is defined to be the renormalized volume; of course, we have not shown so far that $V_+^+$ is independent of the choice
	of $\bar{h} \in [\bar{h}]$ (or equivalently, of $r$).
	
	Since (as we saw above) $Q_{g_{+}} = 6$, the above right-hand side is thus the integral $\int_{X^+} \frac{1}{2}Q_{g_+}dv_{g_+}$. We next turn to the boundary integrals over $Y_{\varepsilon}$ and 
	$M_{\varepsilon}$, beginning with $Y_{\varepsilon}$. We will analyze $\mathcal{L}_Y$ and $T_Y$
	with respect to the metric $g_+$; of course, since $\mathcal{L}_Y$ is a pointwise conformal invariant, it is automatic that the integral of $\mathcal{L}_Y$ over $Y_{\varepsilon}$ will converge as
	$\varepsilon \to 0$. Now, because $g_+$ is Einstein and $Y$ is minimal in $(X,g_+)$, the first and third terms in (\ref{Leq}) vanish in this case. Thus, we get simply
	$\mathcal{L}_Y = -2\mathring{L}_Y^{\alpha\beta}R_{\alpha\beta}^{\tilde{h}} - \tr_{\tilde{h}}\mathring{L}_Y^3$.

	Next turning to $T_Y$, we again compute with respect to the ambient metric $g_+$, i.e., with respect to the non-compactified setting. Again, due to the Einstein condition of $g_+$ and the minimal condition
	on $Y$, the first, second, fourth, sixth, seventh, and eighth terms of (\ref{Teq}) vanish, so we get
	\begin{align*}
		T_Y &= \mathring{L}^{\alpha\beta}R_{\alpha\beta}^{\tilde{h}} + \frac{2}{3}\tr_{\tilde{h}}\mathring{L}_Y^3\\
		&= -\frac{1}{2}\mathcal{L}_Y + \frac{1}{6}\tr_{\tilde{h}}\mathring{L}_Y^3.
	\end{align*}
	Now, $\mathcal{L}_Y$ and $\tr_{\tilde{h}}\mathring{L}_Y^3$ are both pointwise conformal invariants of weight $-3$, so we have exhibited $T_Y$ itself as such a pointwise conformal invariant. We define
	\begin{equation*}
		\mathcal{C}_Y = \frac{1}{2}\mathcal{L}_Y + \frac{1}{6}\tr_{\tilde{h}}\mathring{L}_Y^3.
	\end{equation*}
	This is a pointwise conformal invariant, and the upshot of the above remarks is that 
	\begin{equation}
		\label{Yint}
		\int_{Y_{\varepsilon}} \left( \mathcal{L}_Y + T_Y \right)dv_{\tilde{h}} = \int_{Y_\varepsilon}\mathcal{C}_Ydv_{\tilde{h}} = \int_{\mathring{Y}}\mathcal{C}_Ydv_{\tilde{h}} + O(\varepsilon).
	\end{equation}

	We now turn to the integral over $M_{\varepsilon}^+$ in (\ref{gba}). Here, we will compute $\overline{T}_{M_{\varepsilon}}$ and $\overline{\mathcal{L}}_{M_{\varepsilon}}$, the extrinsic curvature quantities with respect to the
	\emph{compactified} metrics $\bar{g}$ and $\bar{h}_{\varepsilon}$; then we will compute the transformation to $g_+, h_{\varepsilon}$ using equation (\ref{Transeq}), which in particular implies that
	\begin{equation*}
		\int_{M_\varepsilon^+}(\mathcal{L}_M + T_M) dv_{g_+} = \int_{M_{\varepsilon}^+}(\overline{\mathcal{L}}_M + \overline{T}_M + P_3^{\bar{g}}(-\log r))dv_{\bar{g}}.
	\end{equation*}
	Our goal is thus to compute the right-hand side of this equation. We begin by computing some basic quantities. Recalling that $\bar{g} = dr^2 + \bar{h}_r$ and $M_{\varepsilon} = \left\{ r = \varepsilon \right\}$,
	we find that
	\begin{equation*}
		\overline{L}_{M_{\varepsilon}} = -\frac{1}{2}\partial_r\bar{h}_{r}|_{r = \varepsilon} = \varepsilon P^{\bar{h}} + O(\varepsilon^2),
	\end{equation*}
	where $P^{\bar{h}}$ is the Schouten tensor of $\bar{h}$, and we have used (\ref{hbarexp}). Thus,
	\begin{equation}
		\label{HMeq}
		\overline{H}_{M_{\varepsilon}} = \varepsilon (P_{\bar{h}})^{\mu}_{\mu} + O(\varepsilon^3) = \frac{1}{4}\varepsilon R_{\bar{h}} + O(\varepsilon^3).
	\end{equation}
	The reason the error is $O(\varepsilon^3)$ is that the $r^3$ term in the expansion of $\bar{h}_r$ is trace-free. We also have
	\begin{equation*}
		\mathring{\overline{L}}_{M_\varepsilon} = \varepsilon\mathring{P}^{\bar{h}} + O(\varepsilon^2).
	\end{equation*}
	We next wish to compute $R_{\bar{g}}$ on $M_{\varepsilon}$. To do this, we use the fact that $R_{g_+} \equiv -12$ and that $g_+ = r^{-2}\bar{g}$. Thus, we will use the conformal transformation formula for scalar curvature.
	Let $\omega = -\log(r)$. It will be useful to record that
	\begin{equation}
		\label{lapomegeq}
		\Delta_{\bar{g}}\omega = r^{-2} + \frac{1}{4}R_{\bar{h}} + O(r^2),
	\end{equation}
	which follows easily from (\ref{hbarexp}). Thus, from the conformal change formula, we find
	\begin{align*}
		-12 &= r^2(R_{\bar{g}} - 6\Delta_{\bar{g}}\omega - 6|d\omega|_{\bar{g}}^2)\\
		&= r^2\left(R_{\bar{g}} - 6r^{-2} - \frac{3}{2}R_{\bar{h}} - 6r^{-2} + O(r^2)\right),\\
		\intertext{whence}
		R_{\bar{g}} &= \frac{3}{2}R_{\bar{h}} + O(r^2).
	\end{align*}
	We next compute the tracefree tangential Ricci tensor $\mathring{R}_{\mu\nu}^{\bar{g}}$. We will use again the same technique of conformal transformation and the fact that $\ric(g_+) = -3g_+$. We first find
	using (\ref{hbarexp}) that
	\begin{equation*}
		\nabla_{\mu}^{\bar{g}}\nabla_{\nu}^{\bar{g}}\omega = P^{\bar{h}}_{\mu\nu} + O(r).
	\end{equation*}
	It then follows from the equation
	\begin{equation*}
		R_{\mu\nu}^{g_+} = R_{\mu\nu}^{\bar{g}} - 2\nabla_{\mu}^{\bar{g}}\nabla_{\nu}^{\bar{g}}\omega + 2\omega_{\mu}\omega_{\nu} - (\Delta_{\bar{g}}\omega - 2|d\omega|_{\bar{g}}^2)\bar{g}_{\mu\nu}
	\end{equation*}
	that
	\begin{equation*}
		\mathring{R}_{\mu\nu}^{\bar{g}} = 2\mathring{P}^{\bar{h}}_{\mu\nu} + O(r).
	\end{equation*}
	We are ready to analyze the curvature integrands on $M_{\varepsilon}$. First, we easily find using (\ref{Leq}) and the above that
	\begin{equation*}
		\overline{\mathcal{L}}_{M_{\varepsilon}} = O(\varepsilon),
	\end{equation*}
	where the first-order contribution is from the first two terms of (\ref{Leq}), and the last two terms provide contributions of order $O(\varepsilon^3)$. Next, we compute $\overline{T}_{M_{\varepsilon}}$, recalling that
	$\bar{\mu}_{M_{\varepsilon}} = \frac{\partial}{\partial r}$. Then it again follows from the above computations that
	\begin{equation*}
		\overline{T}_{M_{\varepsilon}} = O(\varepsilon).
	\end{equation*}
	The lowest-order contributions come once again from the first three terms of (\ref{Teq}), as well as the sixth.

	We next turn to computing $P_3^{\bar{g}}(\omega) = -P_3^{\bar{g}}(\log(r))$ for $P_3^{\bar{g}}$ associated to $M_{\varepsilon}$. First, observe that $\omega|_{M_{\varepsilon}} \equiv -\log(\varepsilon)$,
	and $\bar{\mu}_{M_{\varepsilon}}(\omega) \equiv \frac{1}{\varepsilon}$. Thus, all tangential derivatives of both quanties vanish, which means the second through fifth terms of (\ref{P3eq}) vanish. Thus,
	only the first and last remain. It follows from (\ref{lapomegeq}) that
	\begin{align*}
		\frac{1}{2}\bar{\mu}_{M_{\varepsilon}}\Delta_{\bar{g}}\omega = -\varepsilon^{-3} + O(\varepsilon).
	\end{align*}
	Next, using again the facts that $R_{\bar{g}} = \frac{3}{2}R_{\bar{h}} + O(r^2)$ and our above calculations, we find that the last term of (\ref{P3eq}) simplifies to
	\begin{equation*}
		\left( \frac{1}{6}R_{\bar{g}} - \frac{1}{2}R_{\bar{h}_{\varepsilon}} - \frac{1}{2}|\mathring{\overline{L}}_{M_{\varepsilon}}|_{\bar{h}_{\varepsilon}}^2 + \frac{1}{3}\overline{H}_{M_{\varepsilon}}^2 \right)
		\bar{\mu}(-\log(r))
		= \frac{1}{4}\varepsilon^{-1} + O(\varepsilon).
	\end{equation*}
	
	Now, we wish to perform the integral over $M_{\varepsilon}^+$, not $M_{\varepsilon}$. Just as for the interior integral, the simplest approach will be first to compute the integral over
	$\left\{ \varepsilon \right\} \times M^+$, and then subtract or add whatever was missed near the corner due to turning of $Y$ away from $\Sigma$. First, we observe that from our above computations,
	it is clear that
	\begin{equation*}
		\int_{M_{\varepsilon}^+}(\overline{T}_{M_{\varepsilon}} + \overline{\mathcal{L}}_{M_{\varepsilon}} + P_3^{\bar{g}}(-\log(r)))dv_{\bar{h}_{\varepsilon}} =
		\int_{M_{\varepsilon}^+}P_3^{\bar{g}}(-\log(r))dv_{\bar{h}_{\varepsilon}} + O(\varepsilon).
	\end{equation*}
	We may focus therefore only on contributions from $P_3^{\bar{g}}(-\log(r))$. We write
	\begin{align*}
		\int_{M_{\varepsilon}^+}P_3^{\bar{g}}(\omega)dv_{\bar{h}_{\varepsilon}} &= \int_{\{\varepsilon\} \times M^+}P_3^{\bar{g}}(\omega)dv_{\bar{h}_{\varepsilon}}\\
		&\quad- \oint_{\Sigma}\int_0^{u(\varepsilon,\zeta)}
		P_3^{\bar{g}}(\omega)(1 + O(w))dwdv_{\bar{k}}(\zeta).
	\end{align*}
	(Compare (\ref{cornervolform}).)
	We compute the first term first. Recall that $dv_{\bar{h}_{\varepsilon}} = (1 - \frac{1}{8}\varepsilon^2 R_{\bar{h}} + O(\varepsilon^4))dv_{\bar{h}}$. Thus,
	\begin{align*}
		\int_{\{\varepsilon\} \times M^+} P_3^{\bar{g}}(\omega) &= \int_{M^+} \left(-\varepsilon^{-3} + \frac{1}{4}\varepsilon^{-1}R_{\bar{h}} + O(\varepsilon)\right)
		\left(1 - \frac{1}{8}\varepsilon^2R_{\bar{h}} + O(\varepsilon^4)\right)dv_{\bar{h}}\\
		&= -\varepsilon^{-3}\vol_{\bar{h}}(M^+) + \frac{3}{8}\varepsilon^{-1}\int_{M^+}R_{\bar{h}}dv_{\bar{h}} + O(\varepsilon).
	\end{align*}
	As for the corner integral, we find using (\ref{uexp})
	\begin{align*}
		\oint_{\Sigma}\int_0^{u(\varepsilon,\zeta)}P_3^{\bar{g}}(\omega)(1 + O(w))dwdv_{\bar{k}}(\zeta) &= \oint_{\Sigma}\left( -\varepsilon^{-3} + O(\varepsilon^{-1}) \right)\cdot\\
			&\quad\cdot\left( \frac{1}{4}\varepsilon^2\overline{\eta}_M + O(\varepsilon^4\log(\varepsilon)) \right)dv_{\bar{k}}\\
			&= -\frac{1}{4}\varepsilon^{-1}\oint_{\Sigma}\overline{\eta}_Mdv_{\bar{k}} + O(\varepsilon\log\varepsilon).
	\end{align*}
	Thus, we have found that
	\begin{equation}
		\label{Mint}
		\begin{split}
		\int_{M_{\varepsilon}^+}(T_{M} + \mathcal{L}_M)dv_{g_+} =& -\varepsilon^{-3}\vol_{\bar{h}}(M^+)\\
		&+ \varepsilon^{-1}\left( \frac{3}{8}\int_{M^+}R_{\bar{h}}dv_{\bar{h}} + \frac{1}{4}\oint_{\Sigma}
		\overline{\eta}_Mdv_{\bar{k}}\right) + o(1).
		\end{split}
	\end{equation}

	We are finally ready to evaluate the corner terms $U_{\Sigma_{\varepsilon}}$ and $G_{\Sigma_{\varepsilon}}$ in (\ref{gba}). Just as for $M_{\varepsilon}^+$, our strategy will be to evaluate first with respect to
	$\bar{g}$, and then use the conformal transformation formula (\ref{Utranseq}) and the pointwise conformal invariance of $G$. Thus, we will find
	\begin{equation*}
		\oint_{\Sigma_{\varepsilon}}(G_k + U_k)dv_k = \oint_{\Sigma_{\varepsilon}} (\overline{G}_{\Sigma_{\varepsilon}} + \overline{U}_{\Sigma_{\varepsilon}} + P_2^{\bar{g}}(-\log r))dv_{\bar{k}_{\varepsilon}}.
	\end{equation*}
	
	To begin, we wish to estimate $\theta_0^{\varepsilon}$, which enters the formulas for $U, G$, and $P_2$. To do this, we find normal vectors $\bar{\mu}_{M_{\varepsilon}}$ and
	$\bar{\mu}_{Y}$. The first is easy: $\bar{\mu}_{M_{\varepsilon}} = \frac{\partial}{\partial r}$. For the second, we observe that, for $\varepsilon$ small, we can write $Y$ as the zero level set
	of $F = w - u(r,\zeta)$ (where, again, $\zeta \in \Sigma$). Now,
	\begin{align*}
		\grad_{\bar{g}}F &= (1 + O(r^2))\frac{\partial}{\partial w} - \frac{\partial u}{\partial r}\frac{\partial}{\partial r} - \bar{k}^{ab}\frac{\partial u}{\partial x^a}\frac{\partial}{\partial x^b} 
		+ O^i(r^3\log (r))\partial_i\\
		&= (1 + O(r^2))\frac{\partial}{\partial w} - \frac{1}{2}r\overline{\eta}_{M}\frac{\partial}{\partial r} - \frac{1}{4}r^2\bar{k}^{ab}\frac{\partial \overline{\eta}_M}{\partial x^a }\frac{\partial}{\partial x^b} + 
		O_{\bar{g}}(r^3\log(r)).
	\end{align*}
	Since $\left|\frac{\partial}{\partial w}\right|_{\bar{g}} = 1 + O(r^2)$, we have
	\begin{equation*}
		|\grad_{\bar{g}}F|_{\bar{g}} = 1 + O(r^2).
	\end{equation*}
	Consequently,
	\begin{equation}
		\label{bmueq}
		\bar{\mu}_Y = \frac{\grad_{\bar{g}}F}{|\grad_{\bar{g}}F|_{\bar{g}}} = (1 + O(r^2))\frac{\partial}{\partial w} - \left( \frac{1}{2}r\overline{\eta}_{M} + O(r^3\log(r)) \right)\frac{\partial}{\partial r} +
		O_{\bar{g}}(r^2).
	\end{equation}
	Thus,
	\begin{equation*}
		\cos(\theta_0^{\varepsilon}) = -\langle \bar{\mu}_{M_{\varepsilon}},\bar{\mu}_Y\rangle = \frac{1}{2}\varepsilon\overline{\eta}_M + O(\varepsilon^3\log(\varepsilon)).
	\end{equation*}
	Next we wish to estimate the second fundamental form $\overline{II}_{Y_{\varepsilon}}$ of $\Sigma_{\varepsilon}$ viewed as a submanifold of $Y_{\varepsilon}$. To do this,
	we first want to know the inward-pointing unit normal vector $\bar{\nu}_{Y_{\varepsilon}}$ to $\Sigma_{\varepsilon}$ in $Y_{\varepsilon}$. By inspection, we can see that
	\begin{equation*}
		V = \frac{\partial}{\partial r} - \frac{\partial F}{\partial r}\frac{\grad_{\bar{g}}F}{|dF|_{\bar{g}}^2}
	\end{equation*}
	is normal to $\Sigma_{\varepsilon}$ and tangent to $Y_{\varepsilon}$, so 
	\begin{equation}
		\label{bnueq}
		\bar{\nu}_{Y_{\varepsilon}} = \frac{V}{|V|_{\bar{g}}} = (1 + O(\varepsilon^2))\frac{\partial}{\partial r} + \frac{1}{2}\varepsilon \overline{\eta}_{M}\frac{\partial}{\partial w} + O(\varepsilon^3\log \varepsilon).
	\end{equation}
	Now, a local frame for $T\Sigma_{\varepsilon}$ is given by $\left\{ X_1,X_2\right\}$, where 
	\begin{equation*}
		X_a = \frac{\partial}{\partial x^a} - \frac{\partial F}{\partial x^a}\frac{\partial}{\partial w}.
	\end{equation*}
	Since $\nabla_{\partial_a}^{\bar{g}}\partial_r = O^i(\varepsilon)\partial_i$ (which is easy to check), we may conclude that $\langle\nabla_{X_a}^{\bar{g}}\bar{\nu}_{Y_{\varepsilon}},X_b\rangle_{\bar{g}} = O(\varepsilon)$.
	Thus, by Weingarten's equation,
	\begin{equation*}
		|\overline{II}_{Y_{\varepsilon}}|_{\bar{g}} = O(\varepsilon).
	\end{equation*}
	
	It now follows that $\overline{G}_{\Sigma_{\varepsilon}} = O(\varepsilon)$: the first term in (\ref{Geq}) because $\cot(\theta_0^{\varepsilon}) = O(\varepsilon)$, and the second because of the estimate
	on $\overline{II}_{Y_{\varepsilon}}$.

	We next turn to $\overline{U}_{\Sigma_{\varepsilon}}$. The second and third terms in (\ref{Ueq}) are $O(\varepsilon)$ for the same reason. Turning to the fourth term, 
	$\bar{\nu}_M\overline{H}_{M_{\varepsilon}} = O(\varepsilon)$ by (\ref{HMeq}). To compute $\bar{\nu}_{Y_{\varepsilon}}\bar{H}_Y$, we first compute $\overline{H}_Y$ using the conformal change
	formula. Recall that $H_Y \equiv 0$. Then again taking $\omega = -\log r$, we find from the conformal transformation formula
	$H_Y = e^{-\omega}(\overline{H}_Y - 3\bar{\mu}_Y(\omega))$ that
	\begin{align*}
		0 &= r(\overline{H}_Y - \frac{3}{2}\overline{\eta}_M + O(r^2\log(r))),\\
		\intertext{whence}
		\overline{H}_Y &= \frac{3}{2}\overline{\eta}_M + O(r^2\log(r)).
	\end{align*}
	Thus, $\bar{\nu}_{Y_{\varepsilon}}\overline{H}_Y = O(\varepsilon\log(\varepsilon))$; so since $\theta_0^{\varepsilon} = \frac{\pi}{2} + O(\varepsilon)$, we have
	\begin{equation*}
		\overline{U}_{\Sigma_{\varepsilon}} = \frac{\pi}{2}K_{\bar{k}} + O(\varepsilon\log\varepsilon).
	\end{equation*}
	Consequently,
	\begin{equation*}
		\oint_{\Sigma_{\varepsilon}}(\overline{G}_{\Sigma_{\varepsilon}} + \overline{U}_{\Sigma_{\varepsilon}})dv_{\bar{k}_{\varepsilon}} = \pi^2\chi(\Sigma) + O(\varepsilon\log\varepsilon).
	\end{equation*}
	We still need to compute the integral of $P_2^{\bar{g}}(-\log r)$. First, still letting $\omega = -\log r$, observe that $\omega|_{M_{\varepsilon}} \equiv -\log \varepsilon$ and that
	$\bar{\mu}_M\omega \equiv -\frac{1}{\varepsilon}$. Thus, the first and second terms of (\ref{P2eq}) in $P_2^{\bar{g}}(\omega)$ vanish identically, as do the terms $\overline{\eta}_{M_{\varepsilon}}
	\bar{\nu}_{M_{\varepsilon}}\omega$, $\overline{\eta}_{Y_{\varepsilon}}\bar{\nu}_{M_{\varepsilon}}\omega$, and $\overline{H}_{M_{\varepsilon}}\bar{\nu}_{M_{\varepsilon}}\omega$.

	Now, the third term takes the form
	\begin{align*}
		\bar{\nu}_{Y_{\varepsilon}}\bar{\mu}_Y\omega &= \bar{\nu}_{Y_{\varepsilon}}\left( \frac{1}{2}\overline{\eta}_M + O(r^2\log(r)) \right)\\
		&= \frac{1}{4}\varepsilon\overline{\eta}_M\partial_w\overline{\eta}_M + O(\varepsilon\log\varepsilon)\\
		&= O(\varepsilon\log\varepsilon).
	\end{align*}

	Next, $\bar{\nu}_{Y_{\varepsilon}}\omega = -\frac{1}{\varepsilon} + O(\varepsilon)$, so $\cot(\theta_0^{\varepsilon})\overline{\eta}_{Y_{\varepsilon}}\bar{\nu}_{Y_{\varepsilon}}\omega = O(\varepsilon)$.
	On the other hand, $-\csc(\theta_0^{\varepsilon})\overline{\eta}_{M_{\varepsilon}}\bar{\nu}_{Y_{\varepsilon}}\omega = \varepsilon^{-1}\overline{\eta}_{M} + O(\varepsilon)$, since
	$\overline{\eta}_{M_{\varepsilon}} = \overline{\eta}_M + O(\varepsilon^2)$ and $\csc(\theta_0^{\varepsilon}) = 1 + O(\varepsilon^2)$.

	Finally,
	\begin{equation*}
		\frac{1}{3}\overline{H}_Y\bar{\nu}_{Y_{\varepsilon}}\omega = -\frac{1}{2}\varepsilon^{-1}\overline{\eta}_M + O(\varepsilon\log\varepsilon).
	\end{equation*}
	Adding together all these terms, we therefore find that $P_2^{\bar{g}}\omega = \frac{1}{2}\varepsilon^{-1}\overline{\eta}_M + O(\varepsilon\log\varepsilon)$. Thus,
	\begin{equation}
		\label{Sigint}
		\oint_{\Sigma_{\varepsilon}}\left( \overline{G}_{\varepsilon} + \overline{U}_{\varepsilon} + P_2^{\bar{g}}(-\log r) \right)dv_{\bar{k}_{\varepsilon}} = \frac{1}{2}\varepsilon^{-1}\oint_{\Sigma}
		\overline{\eta}_Mdv_{\bar{k}} + \pi^2\chi(\Sigma) + O(\varepsilon\log\varepsilon).
	\end{equation}

	Combining (\ref{gba}), (\ref{Xint}), (\ref{Yint}), (\ref{Mint}), and (\ref{Sigint}), we find
	\begin{equation*}
		\pi^2(4\chi(X_{\varepsilon}^+) - \chi(\Sigma)) = 3V_+^+ + \frac{1}{8}\int_{X^+_{\varepsilon}}|W_{g_+}|_{g_+}^2 dv_{g_+} + \int_{Y_{\varepsilon}} \mathcal{C}_Y dv_{\tilde{h}} + O(\varepsilon \log \varepsilon).
	\end{equation*}
	Letting $\varepsilon \to 0$ yields the result.
\end{proof}

\section{Variation of Renormalized Volume}
\label{varsec}

In this section we give a proof of Theorem \ref{varthm}.   Since this will require extensive calculations we begin by establishing some new notational conventions. 

In addition to using the coordinate system
$(r,x^1,x^2,w)$, it will be convenient to use the system $(x^{\tilde{0}},x^{\tilde{1}},x^{\tilde{2}},x^{\tilde{3}}) = (r,x^1,x^2,w - u)$, where $u$ is as in (\ref{wueq}).
We will still use $0 \leq i,j \leq 3$ to refer to coordinate fields on $X$, but will use $0 \leq \talpha,\tbeta \leq 2$ to refer to the coordinate fields tangent to $Y$. It will also be
useful on the interior $\mathring{X}$ to let $x^{\hat{n}}$ be the $g_+$-distance to $\mathring{Y}$, so that $\frac{\partial}{\partial x^{\hat{n}}} = \mu_Y$ is the $g_+$-unit inward normal vector to $\mathring{Y}$.
The system $(r,x^1,x^2,x^{\hat{n}})$ is clearly another coordinate system near $\mathring{Y}$, and the corresponding coordinate vector fields tangent to $Y$ are the same.

As in the introduction, suppose $\mathcal{F}:(-\varepsilon,\varepsilon)_t \times Y \to X$ is a $C^3$ variation of $Y$ through minimal surfaces in $X$ such that $\mathcal{F}(t,\Sigma) \subset M$ for all $t$.  For each $t \in (-\varepsilon, \varepsilon)$, $\mathcal{F}_t(Y)=Y^t$ splits $X$ into two disjoint sets, $X^+_t,$ $X^-_t$ and we can make our choice of $X^+_t$ consistent by fixing a point $p\in X_{0}^+$ and requiring that $p\in X^+_t$ for $t$ in a possibly smaller time interval $t\in (-\delta, \delta).$  Let $V^+_+(t)=V^+_+(X^+_t)$. We will also use the notation $V_+^+(\mathcal{F}_t(Y))$. 
Our goal is to use the formula (\ref{gbrv}) to compute a formula for the first variation, $\dot{V}^+_+.$ 

Before proceeding we recall that strictly speaking, the formula for $V_{+}^{+}$ given by (\ref{gbrv}) only holds for minimal $Y$.  However, as we remarked in the introduction, one can use this formula to define $V_{+}^{+}$ for any dividing hypersurface, in particular for $Y^t = \mathcal{F}_t(Y)$, where $\mathcal{F}_t$ is a general variation of $Y$. 

We begin by making two simplifying assumptions about the variation $\mathcal{F}$.  First, we show that it suffices to consider normal variations of $Y$.  We then weaken the assumption that $Y^t = \mathcal{F}_t(Y)$ is minimal for each $t$, and only assume that minimality is preserved infinitesimally. The latter assumption will suffice to establish the theorem.

To see why it suffices to consider normal variations, let 
$Z = \left.\frac{d}{dt}\mathcal{F}_t\right|_{t = 0}$ be the variation field of $\mathcal{F}$. Write
$Z = Z^{\bot} + Z^{\top}$, with the two uniquely defined fields respectively normal and tangential to $Y$. Now, because $\mathcal{F}_t(\Sigma) \subset M$ for all $t$, along $\Sigma$ we have
$Z^{\top} \in TY \cap TM$, and it follows
that $Z^{\top}$ is tangential to $\Sigma$ along the boundary. Thus, by Theorem 9.34 of Lee and the fact that $\overline{Y}$ is compact, there exists a unique global flow
$\mathcal{G}: \mathbb{R} \times Y \to Y$ such that $\left.\frac{d}{dt}\mathcal{G}\right|_{t = 0} = -Z^{\top}$. Define $\widehat{\mathcal{F}}:(-\varepsilon,\varepsilon) \times Y \to X$
by $\widehat{\mathcal{F}}(t,y) = \mathcal{F}(t,\mathcal{G}(t,y))$. By the chain rule, $\left.\frac{d}{dt}\widehat{\mathcal{F}}_t\right|_{t = 0} = Z^{\bot}$.
On the other hand, $\widehat{\mathcal{F}}_t(Y) = \mathcal{F}_t(Y)$ for all $t$, so it remains a flow through minimal surfaces, and the renormalized volume at each time $t$ is identical.
Thus, it suffices to compute the variation for (initially) normal variation fields, i.e., those
satisfying
\[\left.\frac{d}{dt}\mathcal{F}_t\right|_{t=0} \bot TY.\]

As mentioned, we will also assume
\begin{equation} \label{dHdt} 
\left.\frac{d}{dt} H_{Y^t}\right|_{t=0} = 0, 
\end{equation} 
where $H_{Y^t}$ is the mean curvature of $Y^t$ viewed (via pullback by $\mathcal{F}_t$) as a function
on $Y$.

Let $\mathcal{F} : (-\varepsilon,\varepsilon) \times Y \rightarrow X$, be a $C^3$ normal variation satisfying (\ref{dHdt}).  As in the statement of Theorem \ref{varthm}, we let $f = \left\langle \mu_Y,\left.\frac{d}{dt}\right|_{t = 0}\mathcal{F}\right\rangle_{g_+}$, where $\mu_Y$ is the $(X^+,g_+)$-inward unit normal vector along $Y$.  Since $\mathcal{F}$ is normal, we can write 
\begin{equation} \label{fdef} 
\left.\frac{d}{dt}\right|_{t = 0}\mathcal{F}_t = f\mu_Y. 
\end{equation}
Also, let $\widetilde{\mathcal{F}} = \mathcal{F}|_{(-\varepsilon,\varepsilon) \times \Sigma}$.  Then $\tilde{\mathcal{F}}$ determines $\tilde{f} \in C^{\infty}(\Sigma)$ given by 
\begin{equation} \label{tfdef}
\tilde{f} = \left\langle \left.\frac{d}{dt}\right|_{t = 0}\widetilde{\mathcal{F}}, \bar{\nu}_M\right\rangle, 
\end{equation}
where $\bar{\nu}_M$ is the inward-pointing normal vector to $\Sigma$ in $M^+$ with respect to $\bar{h}$.

From now on, to simplify notation we will let primes denote $\frac{d}{dt}|_{t=0}$. By the formulas (\ref{hdot}), (\ref{Ldot}), and (\ref{Hdot}) in the appendix, the variations of the induced metric, second fundamental form, and mean curvature of $Y$ are given by 
\begin{align}
	\tilde{h}'_{\talpha\tbeta}&=-2fL_{\talpha\tbeta},\label{hprimeeq}\\\nonumber	L'_{\talpha\tbeta}&=\nabla^{\h}_{\talpha}\nabla^{\h}_{\tbeta}f-\tilde{h}^{\gamma\tdelta}L_{\talpha\tgamma}L_{\tbeta\tdelta}f+R^{g_+}_{\talpha \n\tbeta \n}f,\\\nonumber
	H' &= \Delta_{\tilde{h}} f + ( |L_Y|_{\h}^2 - 3)f.
\end{align}
By (\ref{dHdt}), $H' = 0$, so the last formula above implies that $f$ must satisfy 
\begin{equation}\label{JACOBI}
\Delta_{\tilde{h}}f = (3 -  |L_Y|_{\h}^2)f.
\end{equation}

\begin{lemma} \label{fLemma}   $f \in C^{\infty}(\mathring{Y})$ has an asymptotic expansion of the form  
\begin{equation}
	\label{fasympteq}
	f = r^{-1}\tilde{f} + o(1),
\end{equation}
where $\tilde{f} \in C^{\infty}(\Sigma)$ is given by (\ref{tfdef}).   

Conversely, if $|\mathring{L}_Y|_{\h}^2 \leq 3$ on $\mathring{Y}$, then given $\tilde{f} \in C^{\infty}(\Sigma)$, there is a unique solution $f$ to \eqref{JACOBI} satisfying the expansion (\ref{fasympteq}).  
\end{lemma} 

\begin{proof} We first observe that near $M$, 
\begin{align} \label{Lsize} 
 |L_Y|_{\h}^2 = O(r^2). 
\end{align}
This follows from (\ref{2ff}) below, but it can also be seen by using the fact that $L_Y$ is trace-free (since $Y$ is minimal), and the the trace-free second fundamental form is a conformal invariant (of weight $1$).  Using (\ref{Lsize}), it is easy to see that the indicial roots of the operator 
\begin{align*}
	\mathcal{P} = \Delta_{\h} - (3 - |L_{Y}|_{\h}^2)
\end{align*}
are $-1$ and $3$.  It follows that $f$ has an expansion of the form
\begin{equation*}
f = r^{-1} f_{-1} + O(1), 
\end{equation*}
for some $f_{-1} \in C^{\infty}(\Sigma)$. However, using the expansion of the metric $\tilde{h}$ near $M$ in (\ref{hbarexp}), we have $h^{\tilde{0}\tilde{0}} = 1 +O(r^2)$, and using this it is easy to see that 
\begin{equation*}
f - r^{-1} f_{-1} = o(1).
\end{equation*}
as in (\ref{fasympteq}).  Since $\mu_Y = r \bar{\mu}_Y$, (\ref{fdef}) implies 
\begin{align*}
\left.\frac{d}{dt}\right|_{t = 0}\mathcal{F}_t &= f\mu_Y \\
&= \left[ r^{-1} f_{-1} + o(1) \right] \, r \bar{\mu}_Y \\
&= f_{-1} \bar{\mu}_Y + o(r),
\end{align*}
and it follows from (\ref{tfdef}) and the definition of $\tilde{\mathcal{F}}$ that $f_{-1} = \tilde{f}.$

Conversely, given $\tilde{f}$, if we let 
\begin{align*}
f_{-1} = r^{-1} \tilde{f}
\end{align*}
then $\mathcal{P}f_{-1} = O(1)$.  It then follows from standard arguments (see \cite{lee06}) that there is a unique solution of $\mathcal{P}f = 0$ with $f = r^{-1} f_{-1} + O(1)$.  
Again using the expansion of the metric it is readily checked that $f = r^{-1} \tilde{f} + o(1).$   \\
\end{proof}
 
\begin{remark}  Although  $f \in C^{\infty}(\mathring{Y})$, since the indicial roots of the equation satisfied by $f$ are $-1$ and $3$, the expansion of $f$ must in general be expected to have a term $r^3 \log r$, 
	so $rf \in C^{3,\alpha}(\bar{Y})$, and the optimal regularity of $\mathcal{F}$ is $C^3$.
\end{remark}

\begin{proof}[Proof of ~ Theorem \ref{varthm}]   The statement of Theorem \ref{varthm} consists of two claims: the formula for the derivative of $V_{+}^{+}$, and the assertion that $\tilde{f}$ determines $f$.  Since the latter follows from the uniqueness claim in Lemma \ref{fLemma}, to complete the proof of the theorem we just need to carry out the calculation of $\dot{V}_{+}^{+}$.

By Theorem \ref{gbthm},
\begin{align} \nonumber
3 V^{+}_{+}(X_t) = \pi^2(4\chi(X^+_t) - \chi(\partial Y^t)) - \frac{1}{8}\int_{X^{+}_t} |W_{g_+}|_{g_+}^2  dv_{g_+} - \int_{Y^t} \mathcal{C}_{Y^t}  dv_{\h_t}.
\end{align}
We let $\tilde{h}_t = g_+|_{T\mathring{Y}_t}$.
For $\varepsilon > 0$ small, recall that $X_{\varepsilon} = \left\{ x \in X: r(x) > \varepsilon \right\}$. We let $Y^t_{\varepsilon} = Y^t \cap X_{\varepsilon}$, and define
\begin{align}\nonumber
3 V_{\varepsilon}(t) = \pi^2(4\chi(X^+_t\cap X_{\epsilon}) - \chi(\partial Y^t_{\epsilon})) - \frac{1}{8}\int_{X^{+}_t \cap X_{\epsilon}} |W_{g_+}|_{g_+}^2 \, dv_{g_+} - \int_{Y^t_{\varepsilon}} \mathcal{C}_{Y^t} \, dv_{\h_t}.
\end{align}
Then
\begin{align} \nonumber
3 \frac{d}{dt} V_{\varepsilon}(t) \big|_{t=0} = - \frac{1}{8} \frac{d}{dt}  \int_{X^{+}_t \cap X_{\varepsilon}} |W_{g_+}|_{g_+}^2 \, dv_{g_+} \Big|_{t=0} - \frac{d}{dt}  \int_{Y^t_{\varepsilon}} \mathcal{C}_{Y^t} \, dv_{\h_t} \Big|_{t=0}.
\end{align}
For the first integral,
\begin{equation}\label{W}
 -\frac{1}{8}\frac{d}{dt}\bigg|_{t=0} \int_{X^+_t\cap X_{\varepsilon}}|W_{g_+}|_{g_+}^2dv_{g_+}=\frac{1}{8}\int_{Y_{\varepsilon}}|W_{g_+}|_{g_+}^2fdv_{\tilde{h}}.
 \end{equation}
\noindent
To analyze the second integral, we 
let $dv_{\h_t}^{\varepsilon} = \psi dv_{\h_t}$, where $\psi = \theta(r - \varepsilon)$, with $\theta$ the Heaviside function. Then
\begin{align*}
	\left.\frac{d}{dt}\right|_{t = 0}\int_{Y^t_{\varepsilon}}\mathcal{C}_{Y^t}dv_{\h_t} &= \left.\frac{d}{dt}\right|_{t = 0}\int_{Y^t}\mathcal{C}_{Y^t}dv_{\h_t}^{\varepsilon}\\
			&= \lim_{\tau \to 0}\frac{1}{\tau} \left[ \int_{Y} (\mathcal{C}_{Y^{\tau}}\circ \mathcal{F}_{\tau}) (\psi \circ \mathcal{F}_{\tau})
		(\mathcal{F}_{\tau}^*dv_{\h_{\tau}} - dv_{\h})\right.\\
		&\quad	+ \int_Y(\mathcal{C}_{Y^{\tau}}\circ \mathcal{F}_{\tau} - \mathcal{C}_Y)(\psi \circ \mathcal{F}_{\tau}) dv_{\h}\\
		&\quad+ \left.\int_{Y}\mathcal{C}_Y (\psi \circ \mathcal{F}_{\tau} - \psi) dv_{\h} \right]\\
		&= \int_{Y_{\varepsilon}}\mathcal{C}_Y \left( \left.\frac{d}{dt}dv_{\h_t}\right|_{t = 0} \right) \\
			&\quad+ \int_{Y_{\varepsilon}} \left.\frac{d}{dt}\mathcal{C}_{Y^t}\right|_{t = 0}dv_{\h} +
				\lim_{\tau \to 0} \frac{1}{\tau}\int_Y\mathcal{C}_Y (\psi \circ \mathcal{F}_{t} - \psi)dv_{\h}.
\end{align*}
Now by the Implicit Function Theorem, the equation $r(\mathcal{F}(t(p),r(p),\zeta(p))) = \varepsilon$ can be written as $r = \xi(t,\zeta)$ for some smooth
$\xi:(-\delta,\delta) \times \Sigma \to \mathbb{R}$. Let $\bar{k}_{\varepsilon}$ be the
metric induced on $\Sigma_{\varepsilon}$ by $\bar{g}$. Writing $dv_{\h} = \eta r^{-3}drdv_{\bar{k}_{\varepsilon}}$ for some smooth correction factor $\eta$ that is one on $\Sigma_{\varepsilon}$, we may use the fundamental theorem of calculus to write the last term as
\begin{align*}
	\lim_{\tau \to 0}\frac{1}{\tau}\int_{Y}\mathcal{C}_Y(\psi\circ \mathcal{F}_{\tau} - \psi)dv_{\h} &= -\lim_{\tau \to 0}\frac{1}{\tau}\int_{\Sigma_{\varepsilon}}
			\int_{\varepsilon}^{\xi(\tau,\zeta)}\mathcal{C}_Y(r,\zeta)\eta(r,\zeta)r^{-3}drdv_{\bar{k}_{\varepsilon}}(\zeta)\\
			&=-\int_{\Sigma_{\varepsilon}}\left.\frac{d}{dt}\right|_{t = 0}\int_{\varepsilon}^{\xi(t,\zeta)}\mathcal{C}_Y(r,\zeta)\eta(r,\zeta)r^{-3}drdv_{\bar{k}_{\varepsilon}}(\zeta)\\
			&= -\int_{\Sigma_{\varepsilon}}\mathcal{C}_Y(\varepsilon,\zeta)\varepsilon^{-3}\left.\frac{\partial \xi}{\partial t}\right|_{t = 0} dv_{\bar{k}_{\varepsilon}}(\zeta)\\
		&= \int_{\Sigma_{\varepsilon}}\mathcal{C}_Y\varepsilon^{-1}dr(f\mu_Y) dv_{k_{\varepsilon}}\\
		&= \int_{\Sigma_{\varepsilon}}\mathcal{C}_Y\langle r\partial_r,f\mu_Y\rangle_{g_+}dv_{k_{\varepsilon}}.
\end{align*}
Therefore
\begin{align*}
\frac{d}{dt} \int_{Y_{\varepsilon}^t} \mathcal{C}_{Y^t}  dv_{\h_t} \Big|_{t=0} &= \int_{Y_{\varepsilon}} \big( \frac{d}{dt}  \mathcal{C}_{Y^t} \big|_{t=0} \big)   dv_{\h} + \int_{Y_{\varepsilon}} \mathcal{C}_{Y}  \big(
\frac{d}{dt} dv_{\h_t}\big|_{t=0} \big)\numberthis\label{Vp10}\\
&\quad+\int_{\Sigma_{\varepsilon}}\mathcal{C}_{Y}\langle r\partial_r,f\mu_Y\rangle_{g_+} dv_{k_\varepsilon}.
\end{align*}

We dispose of the last term with
\begin{claim}
	$$ \lim_{\varepsilon \to 0}\int_{\Sigma_{\varepsilon}}\mathcal{C}_{Y}\langle r\partial_r,f\mu_Y\rangle_{g_+} dv_{k_{\varepsilon}} = 0.$$
\end{claim}
\begin{proof}
	We know that $\mu_Y=r\bar{\mu}_Y$ and that $\mathcal{C}_{Y}^{g_+}=r^3\mathcal{C}_{Y}^{\bar{g}}.$  
	We also know from (\ref{bmueq}) that
\begin{equation*}
	\langle r\partial_r,\bar{\mu}_Y\rangle_{\bar{g}} =O(\varepsilon^2).
\end{equation*}
So we get $$\mathcal{C}_{Y}^{g_+}\langle r\partial_r, \mu_Y\rangle_{g_+} =r^3\mathcal{C}_{Y}^{\bar{g}}\langle r\partial_r,\mu_Y\rangle_{g_+} =O(\varepsilon^4).$$
Therefore, taking into account the asymptotics of $f$, we get \begin{equation}\label{Vp11}
	\int_{\Sigma_{\varepsilon}}\mathcal{C}_{Y}\langle r\partial_r,f\mu_Y\rangle_{g_+} dv_{k_\varepsilon}=O(\varepsilon).
 \end{equation}\end{proof}

By (\ref{dHdt}) and the formula for the variation of the volume form (\ref{dVdot}) in the appendix we have 
\begin{equation}\label{Vp12}
\frac{d}{dt} dv_{\h_t}\big|_{t=0} = H_{Y} dv_{\h} = 0,
\end{equation}
since $Y$ is minimal.  The minimality of $Y$ to first order also implies $H_{Y^t} = O(t^2)$.  Since $g_{+}$ is Einstein, the formula for $\mathcal{C}_{Y^t}$ thus simplifies to
\begin{equation} \label{Vp13}
	\mathcal{C}_{Y^t} =  -(L_{Y^t})^{\talpha\tbeta} R_{\talpha\tbeta}^{\h_t}
	  -\frac{1}{3}\tr_{\h_t}(L_{Y^t})^3 + O(t^2),
\end{equation}
where $L_{Y^t}$ is the second fundamental form of $Y^t$ with respect to $\mu_Y$ and $R^{\h_t}$ is the Ricci tensor of $\h_t$.  Combining $(\ref{Vp10})$, (\ref{Vp11}), $(\ref{Vp12})$ and $(\ref{Vp13})$ we obtain
\begin{align*} 
\frac{d}{dt} \int_{Y^t_{\varepsilon}}    \mathcal{C}_{Y^t} \, dv_{\h_t} \Big|_{t=0} 
&= -\int_{Y_{\varepsilon}} \frac{d}{dt}  \big( (L^{Y^t})^{\talpha\tbeta} R_{\talpha\tbeta}^{\h_t}   \big) \big|_{t=0} \, dv_{\h} \\
& \quad - \frac{1}{3}  \int_{Y_{\varepsilon}} \frac{d}{dt}  \tr_{\h_t}(L^{Y^t})^3   \big|_{t=0}  \, dv_{\h}+O(\varepsilon).
\end{align*}
We intend to apply integration by parts to the integrand of this expression to write quantities in terms of boundary integrals on $\Sigma$. We first write the integrands in terms of geometric quantities on $Y$.

Define \begin{align*}
 A&=(L_{Y^t})^{\talpha\tbeta}R_{\talpha\tbeta}^{\h_t}\\
 B&=\tr_{\h_t}(L_{Y^t})^3.
 \end{align*}

\noindent
Differentiating A gives
 \begin{align}
 A'&=\h^{\talpha\tgamma}\h^{\tbeta\tdelta}R^{\h}_{\tgamma\tdelta}\nabla_{\talpha}^{\h}\nabla_{\tbeta}^{\h}f+3f(L^2)^{\talpha\tbeta}R^{\h}_{\talpha\tbeta}+f\h^{\talpha\tgamma}
 \h^{\tbeta\tdelta}R^{\h}_{\tgamma\tdelta}R^{g_+}_{\talpha \n\tbeta \n}\\
 &\quad+\h^{\talpha\tgamma}\h^{\tbeta\tdelta}L_{\tgamma\tdelta}(R_{\talpha\tbeta}^{\h})'.\nonumber\end{align}
 \noindent A standard formula for the variation of the Ricci tensor (see e.g. \cite{top06}) gives us
\begin{align}	\label{Vp17}
	(R_{\talpha\tbeta}^{\h})'=&-\frac{1}{2}\big[\Delta_{\h} \tilde{h}'_{\talpha\tbeta}-\nabla^{\h}_{\talpha} (\tdelta_{\tbeta}\tilde{h}')-\nabla^{\h}_{\tbeta}
	(\tdelta_{\talpha}\tilde{h}')+\nabla^{\h}_{\talpha}\nabla^{\h}_{\tbeta}(\tr_{\tilde{h}} \tilde{h}')\big ]\\\nonumber
	&-\tilde{h}^{\tgamma \teta}\tilde{h}^{\tdelta \tzeta} R^{\h}_{\talpha \tgamma \tbeta \tdelta} \tilde{h}'_{\teta \tzeta} +\frac{1}{2}\tilde{h}^{\teta \tzeta}R^{\h}_{\talpha \teta} 
	\tilde{h}_{\tbeta \tzeta}'+\frac{1}{2}\tilde{h}^{\teta\tzeta}R^{\h}_{\tbeta \teta} \tilde{h}_{\talpha \tzeta}'.
 \end{align}
Here $\tdelta$ is the divergence with respect to $\h$. 
Now, by (\ref{hprimeeq}), $\Delta_{\h}\h_{\talpha\tbeta}' = -2\Delta^{\h}(fL_{\talpha\tbeta})$. By the same equation,
\begin{equation*}
	\tr_{\tilde{h}}\tilde{h}'=0.
\end{equation*}
Taking the divergence of both sides of (\ref{hprimeeq}) gives us
\begin{align}\label{Vp16}
	\tdelta_{\tbeta}\tilde{h}'&=\tilde{h}^{\talpha\tgamma}\nabla^{\h}_{\tgamma}\tilde{h}_{\talpha\tbeta}'\\\nonumber
	&=-\nabla_{\h}^{\talpha}(2fL_{\talpha\tbeta})\\\nonumber
	&=-2f\nabla^{\talpha}_{\h}L_{\talpha\tbeta}-2L_{\talpha\tbeta}\nabla_{\h}^{\talpha}f.
\end{align}
Now by Codazzi, we have
  \begin{equation*}
	  R^{g_+}_{\talpha\tbeta\tgamma \n}= \nabla_{\tbeta}^{\h}L_{\talpha\tgamma} -\nabla_{\talpha}^{\h}L_{\tbeta\tgamma}
  \end{equation*}
  along $Y$.
  Contracting $\talpha$ and $\tgamma$ and using the Einstein condition on $g_+$ along with the fact that $Y$ is minimal gives
\begin{equation*}
	0=R^{g_+}_{\tbeta \n}=-\tilde{h}^{\talpha \tgamma}\nabla^{\h}_{\talpha}L_{\tbeta \tgamma}+\tilde{h}^{\talpha\tgamma}\nabla_{\tbeta}^{\h}L_{\talpha\tgamma}= 
	-\nabla_{\h}^{\tgamma}L_{\tbeta \tgamma}+\nabla_{\tbeta}^{\h}H=-\nabla_{\h}^{\talpha}L_{\tbeta\talpha}.
  \end{equation*}
  Hence
\begin{equation}\label{Vp18}
	\nabla^{\talpha}_{\h}L_{\tbeta\talpha}=0
\end{equation}
and
\begin{equation*}
\tdelta_{\tbeta}\tilde{h}'=-2L_{\tbeta \tgamma}\nabla ^{\tgamma}f.
\end{equation*}
Turning to the fifth term of (\ref{Vp17}), we consider the Riemann tensor on $Y$.
As the dimension of $Y$ is three, it follows that the Weyl tensor of $\h$ vanishes, giving us 
\begin{align*}
	R^{\h}_{\talpha \tgamma \tbeta \tdelta} &= \tilde{h}_{\talpha \tbeta}R^{\h}_{\tgamma \tdelta}-\tilde{h}_{\talpha\tdelta}R^{\h}_{\tbeta\tgamma}-
	\tilde{h}_{\tbeta\tgamma}R^{\h}_{\talpha\tdelta}+\tilde{h}_{\tgamma\tdelta}R^{\h}_{\talpha\tbeta}-\frac{1}{2}R^{\h}\tilde{h}_{\talpha\tbeta}\tilde{h}_{\tgamma\tdelta}\\
	&\quad +\frac{1}{2}R^{\h}\tilde{h}_{\talpha\tdelta}\tilde{h}_{\tbeta\tgamma}.	
\end{align*}
Thus,
\begin{align}\nonumber
	-\tilde{h}^{\tgamma\teta}\tilde{h}^{\tdelta \tzeta}R^{\h}_{\talpha\tgamma\tbeta\tdelta}\tilde{h}'_{\teta\tzeta}
&=-R_{\h}^{\tzeta\teta}\tilde{h}_{\teta\tzeta}'\tilde{h}_{\talpha\tbeta}+\tilde{h}^{\teta\tzeta}R^{\h}_{\tbeta\teta}\tilde{h}'_{\talpha\tzeta}+
\tilde{h}^{\teta\tzeta}R^{\h}_{\talpha\teta}\tilde{h}'_{\tbeta\tzeta}-\frac{1}{2} R^{\h}\tilde{h}'_{\talpha\tbeta}.\end{align}
So we can write the last three terms of (\ref{Vp17}) as
\begin{align}\nonumber
	&-\tilde{h}^{\tgamma\teta}\tilde{h}^{\tdelta \tzeta}R^{\h}_{\talpha\tgamma\tbeta\tdelta}\tilde{h}'_{\teta\tzeta}+\frac{1}{2}\h^{\teta\tzeta}R^{\h}_{\talpha\teta}\h'_{\tbeta\tzeta}+
	\frac{1}{2}\h^{\teta\tzeta}R^{\h}_{\tbeta\teta}\h'_{\talpha\tzeta} =\\\nonumber
	& \quad-R^{\h}_{\teta\tzeta}(\h^{\teta\tzeta})'\h_{\talpha\tbeta}+\frac{3}{2}\h^{\teta\tzeta}R^{\h}_{\tbeta\teta}\h'_{\talpha\tzeta}+\frac{3}{2} \h^{\teta\tzeta}R^{\h}_{\talpha\teta}\h'_{\tbeta\tzeta}
	-\frac{1}{2}R^{\h}\h'_{\talpha\tbeta}.
\end{align}
Therefore we have found
\begin{align}\nonumber
	(R^{\h}_{\talpha\tbeta})'=&\Delta^{\h}(fL_{\talpha\tbeta})-\nabla_{\talpha}^{\h}(L_{\tbeta \tgamma}\nabla^{\tgamma}f)-\nabla_{\tbeta}^{\h}(L_{\talpha \tgamma}\nabla^{\tgamma}f)
	+2f(R^{\h}_{\teta \tzeta}L^{\teta \tzeta})\tilde{h}_{\talpha\tbeta}\\\nonumber
&-3fL_{\talpha}{}^{\tgamma}R^{\h}_{\tbeta\tgamma}-3fL_{\tbeta}{}^{\tgamma}R^{\h}_{\talpha\tgamma}+fR^{\h}L_{\talpha\tbeta}.\end{align}
This then lets us write down an expression for $\langle L, (\ric^{\h})' \rangle_{\h}:$
 \begin{equation*}
	 L^{\talpha\tbeta}(R^{\h}_{\talpha\tbeta})'=L^{\talpha\tbeta}\Delta_{\h}(fL_{\talpha\tbeta})-2L^{\talpha\tbeta}\nabla_{\tbeta}^{\h}(L_{\talpha\tgamma}\nabla^{\tgamma}f)
	 -6f(L^2)^{\talpha\tbeta}R^{\h}_{\talpha\tbeta}+fR^{\h}|L|^2;
\end{equation*}
hence 
\begin{align*}
	A'&= R_{\h}^{\talpha\tbeta}\nabla_{\talpha}^{\h}\nabla_{\tbeta}^{\h}f-3f(L^2)^{\talpha\tbeta}R^{\h}_{\talpha\tbeta}+fR_{\h}^{\talpha\tbeta}R^{g_+}_{\talpha \n\tbeta \n}\\
&\hspace{6mm}+L^{\talpha\tbeta}\Delta_{\h}(fL_{\talpha\tbeta})-2L^{\talpha\tbeta}\nabla_{\talpha}^{\h}(L_{\tbeta\tgamma}\nabla_{\h}^{\tgamma}f)+fR^{\h}|L|^2.\end{align*}

Using formula (\ref{Ldot}) in the appendix for the the variation of the second fundamental form, it is straightforward to see that $B'$ is given by
\begin{align*}
B'=(\tr L^3)'&=3(\h^{\talpha\tgamma})'\h^{\tbeta\teta}\h^{\tdelta\tzeta}L_{\talpha\tbeta}L_{\tgamma\tdelta}L_{\teta\tzeta}+3\h^{\talpha\tgamma}\h^{\tbeta\teta}\h^{\tdelta\tzeta}L'_{\talpha\tbeta}L_{\tgamma\tdelta}L_{\teta\tzeta}\\
&=6f|L^2|_{\h}^2+3\h^{\talpha\tgamma}\h^{\tbeta\teta}\h^{\tdelta\tzeta}\big[\nabla^{\h}_{\talpha}\nabla^{\h}_{\tbeta}f-L^2_{\talpha\tbeta}f\\
&\quad+(R^{g_+}_{\talpha \n\tbeta \n}L_{\talpha}^{\tgamma}L_{\tbeta\tgamma})f \big]
L_{\tgamma\tdelta}L_{\teta\tzeta}\\
&=3f|L^2|_{\h}^2+3({\nabla}^{\h}_{\talpha}{\nabla}^{\h}_{\tbeta}f)(L^2)^{\talpha\tbeta}+3f R^{g_+}_{\talpha \n\tbeta \n}(L^2)^{\talpha\tbeta}.
\end{align*}
It will be useful to record two consequences of the Gauss curvature equation.  First, using the Einstein condition, the Ricci curvature of $\h$ can be expressed as  
\begin{equation}\label{Vp22}
R^{\h}_{\teta\tzeta}=-3\h_{\teta\tzeta}-R^{g_+}_{\teta \hat{n}\tzeta \hat{n}}-(L^2)_{\teta\tzeta}.
\end{equation}
It follows that the scalar curvature of $\h$ is given by
\begin{equation}
	 \label{RYeq}
	 R_{\h} =-6-|L|^2.
 \end{equation}
We now focus on rewriting four terms in $A'$ and $B'$ to make them amenable to integration by parts.
We thus make the following definitions:
\begin{align*}
	D_1&=	\int_{Y_{\varepsilon}}R_{\h}^{\talpha\tbeta}\nabla^{\h}_{\talpha}\nabla^{\h}_{\tbeta}fdv_{\h}\\
	D_2&=\int_{Y_{\varepsilon}}L^{\talpha\tbeta}\Delta_{\h} (fL_{\talpha\tbeta})dv_{\h}\\
	D_3&=-\int _{Y_{\varepsilon}}2L^{\talpha\tbeta}\nabla_{\talpha}^{\h}(L_{\tbeta\tgamma}\nabla_{\h}^{\tgamma}f)dv_{\h}\\
	D_4&=\int_{Y_{\varepsilon}}3(L^2)^{\talpha\tbeta}\nabla_{\talpha}^{\h}\nabla_{\tbeta}^{\h}f dv_{\h}.
\end{align*}
We will write each of the above terms as an integral over $Y_{\varepsilon}$ plus an integral over $\Sigma_{\varepsilon}$. 
Recall that $\nu_{Y_{\varepsilon}}$ is the inward pointing $\h$ unit-normal vector field to $\Sigma_{\varepsilon}$ in $Y_{\varepsilon}.$
Integrating by parts then applying the second contracted Bianchi identity and (\ref{RYeq}), we find 
\begin{align*}
	D_1&=\int_{Y_{\varepsilon}}\h^{\talpha\tgamma}\h^{\tbeta\tdelta}R^{\h}_{
	\tgamma\tdelta}\nabla^{\h}_{\talpha}\nabla^{\h}_{\tbeta}fdv_{\h}\\\nonumber
	&=-\int_{Y_{\varepsilon}} \h^{\talpha\tgamma}\h^{\tbeta\tdelta}\nabla_{\talpha}^{\h}R^{\h}_{\tgamma\tdelta}\nabla_{\tbeta}^{\h}fdv_{\tilde{h}}-
	\oint_{\Sigma_{\varepsilon}}R^{\h}_{\talpha\tbeta}\nu_{Y_{\varepsilon}}^{\talpha}\nabla^{\tbeta}_{\h}f dv_{k_{\varepsilon}}\\
	\nonumber&=-\int_{Y_{\varepsilon}}\frac{1}{2}(\nabla^{\talpha}_{\h}R_{\h})\nabla_{\talpha}^{\h}fdv_{\h}-
	\oint_{\Sigma_{\varepsilon}}R^{\h}_{\talpha\tbeta}\nu_{Y_{\varepsilon}}^{\talpha}\nabla^{\tbeta}_{\h}f dv_{k_{\varepsilon}}\\\nonumber
	&=\int_{Y_{\varepsilon}}\frac{1}{2}R_{\h}\Delta_{\h} f dv_{\h}+
	\oint_{\Sigma_{\varepsilon}}\Big[\frac{1}{2}R_{\h}\nu_{Y_{\varepsilon}}(f)-\ric_{\h}(\nu_{Y_{\varepsilon}},\nabla^{\h}f) \Big]dv_{k_{\varepsilon}}\\\nonumber
&=\int_{Y_{\varepsilon}} \bigg(\frac{6+|L|^2}{2}\bigg)(|L|^2-3)f dv_{\h}-
\oint_{\Sigma_{\varepsilon}} \bigg[\ric_{\h}(\nu_{Y_{\varepsilon}},\nabla^{\h}f)- \frac{1}{2}R_{\h}\nu_{Y_{\varepsilon}}(f) \bigg]dv_{k_{\varepsilon}}\\\nonumber
&=\int_{Y_{\varepsilon}}\bigg(\frac{|L|^4}{2} + \frac{3|L|^2}{2} -9\bigg) fdv_{\h}-\oint_{\Sigma_{\varepsilon}}\big(\ric_{\h}-  \frac{1}{2}R_{\h}\h\big)\bigg(\nabla^{\h}f,\nu_{Y_{\varepsilon}}\bigg) dv_{k_{\varepsilon}}.\\\nonumber\\\nonumber
\intertext{Next,}\nonumber
D_2&=\int_{Y_{\varepsilon}}L^{\talpha\tbeta}\Delta_{\h}(fL_{\talpha\tbeta})dv_{\h}\\\nonumber
&=\int_{Y_{\varepsilon}}\bigg[|L|^2\Delta_{\h}f+fL^{\talpha\tbeta}\Delta_{\h}L_{\talpha\tbeta}+2L^{\talpha\tbeta}\nabla^{\tgamma}_{\h}f\nabla_{\tgamma}^{\h}L_{\talpha\tbeta}\bigg]dv_{\h}\\\nonumber
&=\int_{Y_{\varepsilon}}\bigg[|L|^2\Delta_{\h}f+f L^{\talpha\tbeta}\Delta_{\h}L_{\talpha\tbeta}+\langle\nabla^{\h}f ,\nabla^{\h}|L|^2 \rangle\bigg] dv_{\h}\\\nonumber
&=\int_{Y_{\varepsilon}}fL^{\talpha\tbeta}\Delta_{\h}L_{\talpha\tbeta} dv_{\h}-\oint_{\Sigma_{\varepsilon}}|L|^2\nu_{Y_{\varepsilon}}(f)dv_{k_{\varepsilon}}.
\end{align*}
We want to use a Simons-type identity to replace the term $\Delta_{\h} L_{\talpha \tbeta}$.  By the Codazzi equation,
$$R^{g_+}_{\tgamma\talpha\tbeta \n}=\nabla^{\h}_{\talpha}L_{\tgamma\tbeta}-\nabla_{\tgamma}^{\h}L_{\talpha\tbeta},$$
so we may write 
\begin{align}\label{Vp14}
	\h^{\tdelta\tgamma}\nabla_{\tdelta}^{\h}R^{g_+}_{\tgamma\talpha\tbeta \n}&=\h^{\tdelta\tgamma}\nabla_{\tdelta}^{\h}\nabla_{\talpha}^{\h}L_{\tgamma\tbeta}-
	\h^{\tdelta\tgamma}\nabla_{\tdelta}^{\h}\nabla_{\tgamma}^{\h}L_{\talpha\tbeta}\\\nonumber
	&=\nabla_{\h}^{\tgamma}\nabla^{\h}_{\talpha}L_{\tgamma\tbeta}-\Delta_{\h}L_{\talpha\tbeta}.\nonumber
\end{align}
\noindent
Now we want to commute the covariant derivatives in the first term on the right-hand side of this equation. By the Ricci identity,
\begin{align}\nonumber
	\nabla_{\tdelta}^{\h}\nabla_{\talpha}^{\h}L_{\tgamma\tbeta}-\nabla_{\talpha}^{\h}\nabla_{\tdelta}^{\h}L_{\tgamma\tbeta}&=
	R_{\tdelta\talpha\tgamma}^{\h}{}^{\teta}L_{\teta\tbeta}+R_{\tdelta\talpha\tbeta}^{\h}{}^{\teta}L_{\teta\tgamma}.\\\nonumber	
\end{align}
Contracting $\tdelta$ and $\tgamma$ and using $\eqref{Vp18}$ gives
\begin{equation}\label{Vp20}
	\h^{\tdelta\tgamma}\nabla^{\h}_{\tdelta}\nabla^{\h}_{\talpha}L_{\tgamma\tbeta}=\h^{\tdelta\tgamma}R_{\tdelta\talpha\tbeta}^{\h}{}^{\teta}L_{\teta\tgamma}+
	R_{\talpha\tgamma}^{\h}L^{\tgamma}{}_{\tbeta}.
\end{equation}

Combining (\ref{Vp20}) and (\ref{Vp14}), we get 
\begin{equation*}
    \Delta_{\h}L_{\talpha\tbeta} = \h^{\tdelta\tgamma}R^{\h}_{\tdelta\talpha\tbeta}{}^{\teta}L_{\teta\tgamma} + R^{\h}_{\talpha\tgamma}L^{\tgamma}{}_{\tbeta} - \h^{\tdelta\tgamma}\nabla^{\h}_{\tdelta}R^{g_+}_{\tgamma\talpha\tbeta\n}.
\end{equation*}

Therefore,
\begin{align} D_2&=\int_{Y_{\varepsilon}}\bigg[L^{\talpha\tbeta}L^{\tgamma\tdelta}R^{\h}_{\tdelta\talpha \tbeta\tgamma}+(L^2)^{\talpha\tbeta}R^{\h}_{\talpha\tbeta}-L^{\talpha\tbeta}\h^{\tdelta\tgamma}\nabla_{\tdelta}^{\h}R^{g_+}_{\tgamma\talpha\tbeta \n}\bigg]fdv_{\h}\\\nonumber &\hspace{4mm}-
	\oint_{\Sigma_{\varepsilon}}|L|^2\nu_{Y_{\varepsilon}}(f)dv_{k_{\varepsilon}}. \\\nonumber
\end{align}
Applying integration by parts to $D_3$ and using \eqref{Vp18} yields
 \begin{align}
	D_3&=-\int_{Y_{\varepsilon}}2L^{\talpha\tbeta}\nabla_{\talpha}^{\h}(L_{\tbeta\tgamma}\nabla_{\h}^{\tgamma}f)dv_{\h}\\\nonumber
&=\oint_{\Sigma_{\varepsilon}}2L^2(\nabla^{\h} f, \nu_{Y_{\varepsilon}})dv_{k_{\varepsilon}}.\nonumber
	\intertext{Now again using integration by parts and applying \eqref{Vp18} we see}\label{D4}
	D_4&=\int_{Y_{\varepsilon}}3(L^2)^{\talpha\tbeta}\nabla_{\talpha}^{\h}\nabla_{\tbeta}^{\h}f dv_{\tilde{h}}\\\nonumber
&=-\int_{Y_{\varepsilon}}3L^{\talpha\tgamma}\nabla_{\talpha}^{\h}L_{\tgamma}{}^{\tbeta}\nabla_{\tbeta}^{\h}f dv_{\h}-\oint_{\Sigma_{\varepsilon}}3L_{\talpha\tgamma}L^{\tgamma\tbeta}\nu^{\talpha}_{Y_{\varepsilon}}
\nabla_{\tbeta}^{\h}f dv_{k_{\varepsilon}}.\\\nonumber\end{align}
In order to rewrite the first term on the right, we consider the following:
\begin{align}\label{Vp28}
	-3\h^{\tbeta\tdelta}L^{\talpha\tgamma}\nabla_{\talpha}^{\h}L_{\tdelta\tgamma}\nabla_{\tbeta}^{\h}f&=-3\h^{\tbeta\tdelta}L^{\talpha\tgamma}(\nabla_{\talpha}^{\h}L_{\tdelta\tgamma}-\nabla^{\h}_{\tdelta}L_{\talpha\tgamma})
	\nabla^{\h}_{\tbeta}f-3\h^{\tbeta\tdelta}L^{\talpha\tgamma}\nabla_{\tdelta}^{\h}L_{\talpha\tgamma}\nabla^{\h}_{\tbeta}f\\\nonumber
	&=-3L^{\talpha\tgamma}R^{g_+}_{\tbeta\talpha\tgamma \n}\nabla^{\tbeta}_{\h}f-\frac{3}{2}\nabla_{\tbeta}|L|^2\nabla_{\h}^{\tbeta}f.
\end{align}

Using the above formula and then applying integration by parts again we see
\begin{align}\label{D4IBP}
-\int_{Y_{\varepsilon}}3L^{\talpha\tgamma}\nabla_{\talpha}^{\h}L_{\tgamma}{}^{\tbeta}\nabla_{\tbeta}^{\h}f dv_{\h}=&\int_{Y_{\varepsilon}}\left(-3L^{\talpha\tgamma}R^{g_+}_{\tbeta\talpha\tgamma \n}\nabla^{\tbeta}_{\h}f-\frac{3}{2}\nabla_{\tbeta}|L|^2\nabla_{\h}^{\tbeta}f\right) dv_{\h}\\\nonumber
=& 3\int_{Y_{\varepsilon}}\nabla_{\h}^{\tbeta}L^{\talpha\tgamma}R^{g_+}_{\tbeta\talpha\tgamma \n}f  dv_{\h}+3\int_{Y_{\varepsilon}}L^{\talpha\tgamma}\nabla_{\h}^{\tbeta}R^{g_+}_{\tbeta\talpha\tgamma \n}f  dv_{\h}\\\nonumber&+\int_{Y_{\varepsilon}}\frac{3}{2}|L|^2\nabla^{\tbeta}\nabla_{\tbeta}f dv_{\h}  +3\oint_{\Sigma_{\varepsilon}} L^{\talpha\tgamma}\nu_{Y_{\varepsilon}}^{\tdelta}R^{g_+}_{\tdelta\talpha\tgamma \n}f \\\nonumber
&+\frac{3}{2}\oint_{\Sigma_{\varepsilon}}|L|^2\nabla_{\nu_{Y_\varepsilon}}f dv_{k_{\varepsilon}}
\end{align}
We also observe that
\begin{align}\label{Vp24} 3f\nabla_{\h}^{\tbeta}L^{\talpha\tgamma}R^{g_+}_{\tbeta\talpha\tgamma \n} &= \frac{3}{2}f(\nabla_{\tbeta}^{\h}L_{\talpha\tgamma}-\nabla_{\talpha}^{\h}L_{\tbeta\tgamma})
	R_{g_+}{}^{\tbeta\talpha\tgamma}{}_{\n}
	\\\nonumber &= \frac{3}{2}fR^{g_+}_{\talpha\tbeta\tgamma \n}R_{g_+}{}^{\tbeta\talpha\tgamma}{}_{\n}\\\nonumber
	&= -\frac{3}{2}fW^{g_+}_{\talpha\tbeta\tgamma \n}W_{g_+}{}^{\talpha\tbeta\tgamma}{}_{\n}.
\nonumber\end{align}
If we use the above formula to re-write the first term on the right-hand side of \eqref{D4IBP}, and use the resulting formula to re-write \eqref{D4}, we get
\begin{align*}
D_4=&\int_{Y_{\varepsilon}} \bigg[3fL^{\talpha\tbeta}\h^{\tdelta\tgamma}\nabla_{\tdelta}^{\h}R^{g_+}_{\tgamma\talpha\tbeta \n}\\\nonumber
 &\quad-\frac{3}{2}fW^{g_+}_{\talpha\tbeta\tgamma \n}W_{g_+}{}^{\talpha\tbeta\tgamma}{}_{\n}
+\frac{3}{2}|L|^2\Delta_{\h}f\bigg] dv_{\h}\\\nonumber &\quad+\oint_{\Sigma_{\varepsilon}}\bigg[3fL^{\talpha\tgamma}R^{g_+}_{\tbeta\talpha\tgamma \n}\nu^{\tbeta}_{Y_{\varepsilon}}
+\frac{3}{2}|L|^2\nu_{Y_{\varepsilon}}(f)-3L^2(\nabla^{\h} f, \nu_{Y_{\varepsilon}})\bigg]dv_{k_{\varepsilon}}.\end{align*}
It is interesting to note that the cancellation of the first term in $D_4$ with the last
interior term of $D_2$ accounts for the absence of any derivatives of Weyl terms in our
final formula.

Now we want to compute $\int_{Y_{\varepsilon}}\mathcal{C}'dv_{\h}=
-\int_{Y_{\varepsilon}} A'dv_{\h} -\frac{1}{3}\int_{Y_{\varepsilon}} B' dv_{\h}.$
Using our expressions for $D_1,D_2,D_3$ and $D_4$ and gathering together all of the terms that appear as integrals over $Y_{\varepsilon}$ we get:
\begin{align}
	I_{Y} &:= -\int_{Y_{\varepsilon}} \bigg[\bigg(\frac{3|L|^2}{2}+\frac{|L|^4}{2}-9\bigg)f\\\nonumber
&\quad+L^{\talpha\tbeta}L^{\tgamma\tdelta}R^{\h}_{\tgamma\talpha \tbeta\tdelta}f+R^{\h}_{\talpha\tbeta}(L^2)^{\talpha\tbeta}f-L^{\talpha\tbeta}\h^{\tdelta\tgamma}\nabla_{\tdelta}^{\h}R^{g_+}_{\tgamma\talpha\tbeta \n}f&\\ \nonumber
&\quad+fL^{\talpha\tbeta}\h^{\tdelta\tgamma}\nabla_{\tdelta}^{\h}R^{g_+}_{\tgamma\talpha\tbeta \n}\\\nonumber
 &\quad-\frac{1}{2}fW^{g_+}_{\talpha\tbeta\tgamma \n}W_{g_+}{}^{\talpha\tbeta\tgamma}{}_{\n}
+\frac{1}{2}|L|^2\Delta_{\h}f&\\\nonumber
  &\quad-3f(L^2)^{\talpha\tbeta}R^{\h}_{\talpha\tbeta}+fR_{\h}^{\talpha\beta}R^{g_+}_{\talpha \n\tbeta \n}+fR^{\h}|L|^2\\\nonumber
  &\quad+f|L^2|_{\h}^2+f R^{g_+}_{\talpha \n\tbeta \n}(L^2)^{\talpha\tbeta}\bigg] dv_{\h}\\\nonumber
&=
	-\int_{Y_{\varepsilon}} \bigg[\bigg(\frac{3|L|^2}{2}+\frac{|L|^4}{2}-9\bigg)f\label{Vp25}\\
		&\quad+L^{\talpha\tbeta}L^{\tgamma\tdelta}R^{\h}_{\tgamma\talpha \tbeta\tdelta}f&\\ \nonumber
 &\quad-\frac{1}{2}fW^{g_+}_{\talpha\tbeta\tgamma \n}W_{g_+}{}^{\talpha\tbeta\tgamma}{}_{\n}
+\frac{1}{2}|L|^2\Delta_{\h}f&\\\nonumber
  &-2f(L^2)^{\talpha\tbeta}R^{\h}_{\talpha\tbeta}+fR_{\h}^{\talpha\beta}R^{g_+}_{\talpha \n\tbeta \n}+fR^{\h}|L|^2\\\nonumber
  &\quad+f|L^2|_{\h}^2+f R^{g_+}_{\talpha \n\tbeta \n}(L^2)^{\talpha\tbeta}\bigg] dv_{\h}.\end{align}
  Next, decomposing the Riemann tensor of $g_+$ gives 
\begin{equation}
	\label{ricweq}
	R^{g_+}_{\talpha \hat{n}\tbeta \hat{n}}=W^{g_+}_{\talpha \hat{n}\tbeta \hat{n}}-\h_{\talpha\tbeta}.
\end{equation}
Applying \eqref{ricweq} to \eqref{Vp22} gives
\begin{equation}\label{Ric}
	R^{\h}_{\talpha\tbeta}= -L^2_{\talpha\tbeta}-2\h_{\talpha\tbeta}-W^{g_+}_{\talpha \hat{n}\tbeta \hat{n}}.
\end{equation}
Decomposing the Riemann tensor of $\h$ allows us to write
\begin{align}\label{Rich}
	R^{\h}_{\tdelta\talpha\tbeta\tzeta}L^{\tzeta\tdelta}L^{\talpha\tbeta}&=L^{\tzeta\tdelta}L^{\talpha\tbeta}[\h_{\tdelta\tbeta}R^{\h}_{\talpha\tzeta}-
		\h_{\tdelta\tzeta}R^{\h}_{\talpha\tbeta}-\h_{\talpha\tbeta}R^{\h}_{\tdelta\tzeta}+\h_{\talpha\tzeta}R^{\h}_{\tdelta\tbeta}\\\nonumber
	& \hspace{6.5mm}-\frac{1}{2}R_{\h}\h_{\tdelta\tbeta}\h_{\talpha\tzeta}+\frac{1}{2}R_{\h}\h_{\tdelta\tzeta}\h_{\talpha\tbeta}].\\\nonumber
\intertext{Next we apply \eqref{Ric} and then \eqref{RYeq} to get}	\nonumber
R^{\h}_{\tdelta\talpha\tbeta\tzeta}L^{\tzeta\tdelta}L^{\talpha\tbeta}&=2(L^2)^{\talpha\tbeta}R^{\h}_{\talpha\tbeta}-
	\frac{1}{2}|L|^2R_{\h}\\\nonumber
	&=-|L|^4-4|L|^2-2(L^2)^{\talpha\tbeta}W^{g_+}_{\talpha\hat{n}\tbeta\hat{n}}+3|L|^2+\frac{1}{2}|L|^4.
\end{align}
Note that we also use here the fact that $|L^2|^2=\frac{1}{2}|L|^4,$ which holds because $H_Y = 0$.

Simplifying gives
\begin{align}L^{\talpha\tbeta}L^{\tgamma\tdelta}R^{\h}_{\tgamma\talpha \tbeta\tdelta}&=-|L|^2-\frac{1}{2}|L|^4-2(L^2)^{\talpha\tbeta}W^{g_+}_{\talpha\hat{n}\tbeta\hat{n}}.\nonumber
\end{align}
Applying this to re-write $L^{\talpha\tbeta}L^{\tgamma\tdelta}R^{\h}_{\tgamma\talpha \tbeta\tdelta}$ and using \eqref{JACOBI} to re-write $\Delta^{\h}f$ gives
\begin{align}\nonumber
	I_Y = -\int_{Y_{\varepsilon}}& \bigg[\bigg(\frac{3|L|^2}{2}+\frac{|L|^4}{2}-9\bigg)f\\\nonumber
		&-|L|^2f-2(L^2)^{\talpha\tbeta}W^{g_+}_{\talpha\hat{n}\tbeta\hat{n}}f-\frac{1}{2}|L|^4f&\\ \nonumber
 &-\frac{1}{2}fW^{g_+}_{\talpha\tbeta\tgamma \n}W_{g_+}{}^{\talpha\tbeta\tgamma}{}_{\n}
+\frac{3}{2}|L|^2f-\frac{1}{2}|L|^4f&\\\nonumber
&+2f|L^2|^2+4|L|^2f+2(L^2)^{\talpha\tbeta}W^{g_+}_{\talpha\hat{n}\tbeta\hat{n}}-6f|L|^2-f|L|^4\\\nonumber
&-(L^2)^{\talpha\tbeta}W_{\talpha\hat{n}\tbeta\hat{n}}f-W_{\talpha\hat{n}\tbeta\hat{n}}W^{\talpha}{}_{\hat{n}}{}^{\tbeta}{}_{\hat{n}}f+|L|^2f+6f\\\nonumber
  &+f|L^2|_{\h}^2+f W^{g_+}_{\talpha \n\tbeta \n}(L^2)^{\talpha\tbeta}-f|L|^2\bigg] dv_{\h}\\\label{int}
  =&\int_{Y_{\varepsilon}}\left( 3f+\frac{1}{2}fW_{\talpha\tbeta\tgamma\hat{n}}W^{\talpha\tbeta\tgamma}{}_{\hat{n}}+fW_{\talpha\hat{n}\tgamma\hat{n}}W^{\talpha}{}_{\hat{n}}{}^{\tgamma}{}_{\hat{n}}\right)dv_{\h}.
\end{align}
We may simplify this helpfully:
\begin{claim}
\begin{equation}
   \frac{1}{2}W^{g_+}_{\talpha\tbeta\tgamma \n}W_{g_+}{}^{\talpha\tbeta\tgamma}{}_{\n}+W^{g_+}_{\talpha \n\tbeta \n}W_{g_+}^{\talpha \n\tbeta \n}=\frac{1}{8}|W_{g_+}|^2_{g_+}.
\end{equation}
\end{claim}
\begin{proof}
	Observe that
\begin{align*}
|W_{g_+}|_{g_+}^2&=W^{g_+}_{ijkl}W_{g_+}^{ijkl}\\
&=4W^{g_+}_{\hat{n}\talpha\tbeta\tgamma}W_{g_+}^{\hat{n}\talpha\tbeta\tgamma}+4W^{g_+}_{\hat{n}\talpha\hat{n}\tbeta}W_{g_+}^{\hat{n}\talpha\hat{n}\tbeta}+W^{g_+}_{\talpha\tbeta\tgamma\tdelta}W_{g_+}^{\talpha\tbeta\tgamma\tdelta}.\\
\end{align*}
Now, $W^{g_+}_{\talpha\tbeta\tgamma\tdelta}$ is an algebraic curvature tensor on
$Y$, a three-manifold, and (omitting $g_+$ for clarity) its trace is given by
\begin{equation*}
    W_{\talpha\tbeta}{}^{\talpha}{}_{\tdelta} = -W_{\n\tbeta}{}^{\n}{}_{\tdelta}.
\end{equation*}
But (by, e.g., Prop. 7.23 and Corollary 7.25 of \cite{lee13}), an algebraic curvature tensor
on a three-manifold is determined by its trace; in this case, the formula reads
\begin{equation*}
    W_{\talpha\tbeta\tgamma\tdelta} = W_{\n\talpha}{}^{\n}{}_{\tdelta}\h_{\tbeta\tgamma}
    + W_{\n\tbeta}{}^{\n}{}_{\tgamma}\h_{\talpha\tdelta} - W_{\n\talpha}{}^{\n}{}_{\tgamma}
    \h_{\tbeta\tdelta} - W_{\n\tbeta}{}^{\n}{}_{\tdelta}\h_{\talpha\tgamma}.
\end{equation*}
It follows that
\begin{equation*}
    W_{\talpha\tbeta\tgamma\tdelta}^{g_+}W_{g_+}^{\talpha\tbeta\tgamma\tdelta} =
    4W_{g_+}^{\n\talpha\n\tbeta}W^{g_+}_{\n\talpha\n\tbeta}.
\end{equation*}
So
\begin{equation*}
    |W_{g_+}|_{g_+}^2 = 4W^{g_+}_{\hat{n}\talpha\tbeta\tgamma}W_{g_+}^{\hat{n}\talpha\tbeta\tgamma}+8W^{g_+}_{\hat{n}\talpha\hat{n}\tbeta}W_{g_+}^{\hat{n}\talpha\hat{n}\tbeta}.
\end{equation*}
\end{proof}

It follows from the previous claim and (\ref{JACOBI}) that \eqref{int} is equal to
\begin{equation}\label{Vp27}
	I_Y = \int_{Y_{\varepsilon}}\bigg[|L|^2f+\frac{1}{8}|W_{g_+}|_{g_+}^2f\bigg] dv_{\h}-\oint_{\Sigma_{\varepsilon}} \nu_{Y_{\varepsilon}}(f) dv_{k_{\varepsilon}}.
\end{equation}

Gathering the boundary terms from $D_1$, $D_2$, $D_3$ and $D_4$ and the normal derivative term on the above line we get
\begin{align}\label{Vp26}
	\oint_{\Sigma_{\varepsilon}} &\bigg[\bigg(\ric_{\h}-\frac{1}{2}R_{\h}\tilde{h}\bigg)(\nabla^{\h}f, \nu_{Y_{\varepsilon}})+|L|^2\nu_{Y_{\varepsilon}}(f)-2L^2(\nabla^{\h} f, \nu_{Y_{\varepsilon}}) \\\nonumber
	&-fL^{\talpha\tgamma}R^{g_+}_{\tbeta\talpha\tgamma \n}\nu_{Y_{\varepsilon}}^{\tbeta}-\frac{1}{2}|L|^2\nu_{\varepsilon}(f)+L^2(\nabla f, \nu_{\varepsilon})-\nu_{Y_{\varepsilon}}(f)\bigg]dv_{k_{\varepsilon}}.
\end{align}
Now we apply \eqref{Ric} to the first term and use 
$R^{g_+}_{\talpha\tbeta\tgamma \n}=W^{g_+}_{\talpha\tbeta\tgamma \n}$
to re-write $fL^{\talpha\tgamma}R^{g_+}_{\tbeta\talpha\tgamma \n}\nu_{\varepsilon}^{\tbeta},$ giving us
 \begin{align}\nonumber\oint_{\Sigma_{\varepsilon}}  \bigg[-2L^2(\nabla f, \nu_{Y_{\varepsilon}}) + fL^{\talpha\tgamma}W^{g_+}_{\talpha\tbeta\tgamma \n}
	 \nu_{Y_{\varepsilon}}^{\tbeta}-W^{g_+}_{\talpha \n \tbeta \n} f^{\talpha} \nu_{\varepsilon}^{\tbeta}+|L|^2\nu_{Y_{\varepsilon}}(f)\bigg] dv_{k_{\varepsilon}}. \\\nonumber
\end{align}
Combining this with \eqref{W}, \eqref{Vp10} and $\eqref{Vp27}$ gives us
\begin{align}\nonumber
	3\frac{d}{dt}V_{\epsilon}(t) \big|_{t=0} &=\int_{Y_{\varepsilon}}|L|^2fdv_{\h}+\oint_{\Sigma_{\varepsilon}} \bigg[-W^{g_+}_{\talpha \n\tbeta \n}\nu_{Y_{\varepsilon}}^{\talpha}\nabla_{\h}^{\tbeta}f+|L|^2\nu_{Y_{\varepsilon}}(f)
		\\\label{deriveq}
		&\quad-2L_{\talpha\tgamma}L^{\tgamma}{}_{\tbeta}\nu_{Y_{\varepsilon}}^{\talpha}\nabla_{\h}^{\tbeta}f
	+L^{\tgamma\tdelta}W^{g_+}_{\tgamma\talpha\tdelta \n}f\nu_{Y_{\varepsilon}}^{\talpha}\bigg] dv_{k_{\varepsilon}} + O(\varepsilon).
\end{align}
Next we will examine the asymptotics of the term $-W^{g_+}_{\talpha \n \tbeta \n}\nu^{\talpha}_{Y_{\varepsilon}}f^{\tbeta}.$
Now, it follows from (\ref{hbarexp}), (\ref{scheq}), and the second-last equation on the bottom of page 52 of \cite{fg12} that
\begin{equation*}
	W^{\bar{g}}_{0\mu0\nu} = O(r).
\end{equation*}
Moreover, from the first equation on p. 53 of the same book, we may conclude that
\begin{equation*}
	W^{\bar{g}}_{0\mu0\nu} = -\frac{3}{2}rg^{(3)}_{\mu\nu} + O(r^2)
\end{equation*}
with $g^{(3)}$ as in (\ref{hbarexp}). By the conformal change formula for the Weyl tensor, therefore, we find
\begin{equation}\label{Weylasymp}
	W^{g_+}_{\mu0\nu0}=-\frac{3}{2}r^{-1}g^{(3)}_{\mu \nu}+O_{\bar{g}}(1).
\end{equation}
Now by (\ref{fasympteq}),
\begin{align*}
-W_{g_+}(\nu_{Y_{\varepsilon}},\mu_Y,\nabla^{\h}f,\mu_Y)&=-r^3W_{g_+}(\bar{\nu}_{Y_{\epsilon}},\bar{\mu}_Y,\nabla^Yf,\bar{\mu}_Y) \\
	&=-r^3W^{g_+}_{\talpha \bar{n} \tbeta \bar{n}}\bar{\nu}_{Y_{\epsilon}}^{\talpha}f^{\tbeta} \\
	&=-r^5W^{g_+}_{\talpha \bar{n} \tbeta \bar{n}}\bar{\nu}_{Y_{\epsilon}}^{\talpha}\bar{g}^{\tbeta\tgamma}\partial_{\tgamma}f \\
	&=r^3W_{\tilde{0}\bar{n}\tilde{0}\bar{n}}\bar{\nu}_{Y_{\epsilon}}^{\tilde{0}} \tilde{f}+O(r^4),
\end{align*}
where $\bar{n}$ corresponds to $\bar{\mu}_Y$.
Taking (\ref{bmueq}), (\ref{bnueq}), $\eqref{Weylasymp},$ and (\ref{fasympteq}), we see that the first corner term of (\ref{deriveq}) may be written
\begin{equation}
	\label{Wtermeq}
	\oint_{\Sigma_{\varepsilon}} -W_{g_+}(\nu_{Y_{\epsilon}},\mu_Y,\nabla^{\h}f,\mu_Y) dv_{k_{\varepsilon}}=
	\oint_{\Sigma} \frac{3}{2}g^{(3)}(\bar{\nu}_M,\bar{\nu}_M)\tilde{f} dv_{\bar{k}}+O(\varepsilon).
\end{equation}

We now simplify the remaining terms of (\ref{deriveq}):
\begin{claim}
	\label{lastclaim}
\begin{align}\nonumber
\int_{Y_{\varepsilon}}|L|^2fdv_{\h}&+\oint_{\Sigma_{\varepsilon}} \bigg[|L|^2\nu_{Y_{\varepsilon}}(f)-2L_{\talpha\tgamma}L^{\tgamma}{}_{\tbeta}\nu_{Y_{\varepsilon}}^{\talpha}\nabla_{\h}^{\tbeta}f
		\\
		&\quad
	+L^{\tgamma\tdelta}W^{g_+}_{\tgamma\talpha\tdelta \n}f\nu_{Y_{\varepsilon}}^{\talpha}\bigg] dv_{k_{\varepsilon}}\\\nonumber
&=f.p. \int_{\mathring{Y}}|L|^2f dv_{\h}+O(\varepsilon \log\varepsilon).
\end{align}

\end{claim}
\begin{proof}
	Observe that
\begin{equation}\label{2ff}
    L_{\tilde{\alpha}\tilde{\beta}}=\frac{\overline{L}_{\tilde{\alpha}\tilde{\beta}}}{r}+\frac{\bar{\mu}_{Y}(r)}{r^2}\bar{g}_{\tilde{\alpha}\tilde{\beta}}.
\end{equation}
Now 
\begin{equation}
	\bar{\mu}_Y=(1+O(r^2))\partial_w-\left(\frac{\overline{\eta}_Mr}{2}+O(r^3\log r)\right)\partial_r+O^a(r^2)\partial_{a}.
\end{equation}
Therefore 
\begin{equation}
    \bar{\mu}_Y(r)=-\frac{1}{2}\big[\overline{\eta}_Mr+O(r^3\log r)\big].
\end{equation}
Now using the fact that $\bar{g}_{\tilde{\alpha}\tilde{\beta}}=\bar{g}_{\alpha\beta}+O(r^2)$ we may write
\begin{equation}
    \overline{L}_{\tilde{\alpha}\tilde{\beta}}=-\frac{1}{2}\bar{\mu}_Y\bar{g}_{\alpha\beta}+O(r^2)=-\frac{1}{2}\partial_w\bar{g}_{\alpha\beta}+O(r^2).
\end{equation}
Therefore we may write
\begin{equation}\label{L}
    L_{\tilde{\alpha}\tilde{\beta}}=-\frac{\partial_w\bar{g}_{\alpha\beta}}{2r}-\frac{\overline{\eta}_M\bar{g}_{\alpha\beta}}{2r}+O(r\log r).
\end{equation}
Hence
\begin{align}
|L|^2_{\h}\nu_{Y_{\varepsilon}}(f)dv_{k_{\varepsilon}}&=\varepsilon^4\bar{g}^{\alpha\gamma}\bar{g}^{\beta\delta}\left[ \frac{\partial_w\bar{g}_{\alpha\beta}}{2\varepsilon}+\frac{\overline{\eta}_M\bar{g}_{\alpha\beta}}{2\varepsilon}+O(\varepsilon \log\varepsilon)\right]\cdot\\\nonumber
&\quad\cdot\left[ \frac{\partial_w\bar{g}_{\gamma\delta}}{2\varepsilon}+\frac{\overline{\eta}_M\bar{g}_{\gamma\delta}}{2\varepsilon}+O(\varepsilon \log\varepsilon)\right]\left[-\tilde{f}\varepsilon^{-1}+O(\varepsilon)
\right]dv_{k_{\varepsilon}}\\\nonumber
&=\big[-\varepsilon^{-1} |\mathring{\overline{II}}_{M}|^2_{k}\tilde{f}+O(\varepsilon \log\varepsilon)\big]dv_{\bar{k}_{\varepsilon}},
\end{align}
so we may write 
\begin{equation*}
	\oint_{\Sigma_{\varepsilon}}|L|^2_{\h}\nu_{Y_{\varepsilon}}(f)dv_{k_{\varepsilon}}= -\oint_{\Sigma} |\mathring{\overline{II}}_M|^2_{k}\tilde{f}dv_{k}\varepsilon^{-1}+O(\varepsilon \log\varepsilon).
\end{equation*}
Now,
\begin{align}\label{LL}
|L|^2_{\h} fdv_{\h}&=r^4\bar{g}^{\alpha\gamma}\bar{g}^{\beta\delta}\left[ \frac{\partial_w\bar{g}_{\alpha\beta}}{2r}+\frac{\overline{\eta}_M\bar{g}_{\alpha\beta}}{2r}+O(r \log r)\right]\cdot\\
\nonumber&\quad \cdot\left[ \frac{\partial_w\bar{g}_{\gamma\delta}}{2r}+\frac{\overline{\eta}_M\bar{g}_{\gamma\delta}}{2r}+O(r \log r)\right]\left[\tilde{f}r^{-1}+O(r) \right]dv_{\h}\\\nonumber
&\hspace{20mm}=\big[r^{-2} |\mathring{\overline{II}}_M|^2_{k}\tilde{f}+O( \log r)\big]dv_{\bar{\h}},
\end{align}
so 
\begin{align}\label{LLL}
\int_{Y_{\varepsilon}} |L|^2_{\h} fdv_{\h}&=C+\int_{\varepsilon}^{r_0} \oint_{\Sigma}|L|^2_{\h} fdv_{k} dr \\\nonumber
&= C+\int_{\varepsilon}^{r_0} \oint_{\Sigma}r^{-2} |\mathring{\overline{II}}_{M}|^2_{k}\tilde{f}+O(\log r)dv_{k} dr\\\nonumber
&=C' +\varepsilon^{-1}\oint_{\Sigma} |\mathring{\overline{II}}_M|^2_k\tilde{f}dv_{k}+O(\varepsilon \log\varepsilon)
\end{align}
for some constants $C$ and $C'$ and $r_0>0$ chosen small enough. Observe that
\begin{equation*}
	C' = f.p. \int_{\mathring{Y}}|L|^2fdv_{\h}.
\end{equation*}
By \eqref{L} we can write
\begin{equation*}
 (L^2)_{\tilde{\alpha}\tilde{\beta}}=O(1),
\end{equation*}
\begin{equation*}
    (L^2)_{\tilde{\alpha}\tilde{0}}=O(r).
\end{equation*}
Also observe by \eqref{bnueq} that $\nu_{\varepsilon}^{\talpha}=O(r^2)$ unless $\talpha=\tilde{0},$ in which case $\nu_{\varepsilon}^{\tilde{0}}=O(r).$
Now if we let $a$ run over the indices $1,2$ we can write 
\begin{align*}
	L_{\tilde{\alpha}\tilde{\gamma}}L^{\tilde{\gamma}}_{\tilde{\beta}}\nu_{\varepsilon}^{\talpha}\nabla^{\tilde{\beta}}f&=h^{\tbeta\tdelta}(L^2)_{\tilde{a}\tbeta}\nu_{\varepsilon}^{\tilde{a}}f_{\tdelta}+h^{\tbeta\tdelta}(L^2)_{\tilde{0}\tbeta}\nu_{\varepsilon}^{\tilde{0}}f_{\tdelta}\\ 
    &=O(r^3)
\end{align*}
It follows that  
\begin{equation*}
	L_{\tilde{\alpha}\tilde{\gamma}}L^{\tilde{\gamma}}_{\tilde{\beta}}\nu_{\varepsilon}^{\tilde{\alpha}}\nabla^{\tilde{\beta}}f dv_k= O(\varepsilon) dv_{\bar{k}},
\end{equation*}
so
$$\oint_{\Sigma_{\varepsilon}} L_{\tilde{\alpha}\tilde{\gamma}}L^{\tilde{\gamma}}_{\tilde{\beta}}\nu_{\varepsilon}^{\tilde{\alpha}}\nabla^{\tilde{\beta}}f dv_k=O(\varepsilon). $$

Now we turn our attention to the term $L^{\tilde{\gamma}\tilde{\delta}}W^{g_+}_{\gamma\alpha\delta n}f\nu_{\varepsilon}^{\alpha}.$ First observe that
\begin{equation*}
	W^{g_+}_{\tilde{\gamma}\tilde{\alpha}\tilde{\delta} \hat{n}}=rW^{g_+}_{\tilde{\gamma}\tilde{\alpha}\tilde{\delta} \hat{\bar{n}}}
\end{equation*}
and
\begin{equation*}
    W^{g_+}_{\tilde{\gamma}\tilde{\alpha}\tilde{\delta} \hat{\bar{n}}}= \frac{W^{\bar{g}}_{\tilde{\gamma}\tilde{\alpha}\tilde{\delta} \hat{\bar{n}}}}{r^2},
\end{equation*}
where $\hat{\bar{n}}$ corresponds to $\bar{\mu}_Y.$ 
Now 
\begin{align*}
W^{\bar{g}}_{\tilde{\gamma} \tilde{0} \tilde{\delta} \bar{\hat{n}}}&=R^{\bar{g}}_{\tilde{\gamma} \tilde{0} \tilde{\delta} \hat{\bar{n}}}\\
&=\nabla^{\bar{g}}_{\tilde{0}} L^{\bar{g}}_{\tilde{\gamma} \tilde{\delta}} - \nabla^{\bar{g}}_{\tilde{\gamma}} L^{\bar{g}}_{\tilde{0}\tilde{\delta}}\\
&=\partial_{\hat{0}} \overline{L}_{\tilde{\gamma}\tilde{\delta}}-\Gamma_{\tilde{0} \tilde{\gamma}}^{\tilde{\beta}} \overline{L}_{\tilde{\beta} \tilde{\delta}}-\Gamma_{\tilde{0} \tilde{\delta}}^{\tilde{\beta}} \overline{L}_{\tilde{\beta} \tilde{\gamma}}-( \partial_{\tilde{\gamma}} \overline{L}_{\tilde{0}\tilde{\delta}}-\Gamma_{\tilde{0} \tilde{\gamma}}^{\tilde{\beta}} \overline{L}_{\tilde{\beta} \tilde{\delta}}-\Gamma_{\tilde{\gamma} \tilde{\delta}}^{\tilde{\beta}} \overline{L}_{\tilde{\beta} \tilde{0}} )\\
&=\partial_{\tilde{0}} \overline{L}_{\tilde{\gamma}\tilde{\delta}}-\Gamma_{\tilde{0} \tilde{\delta}}^{\tilde{\beta}} \overline{L}_{\tilde{\beta} \tilde{\gamma}}-( \partial_{\tilde{\gamma}} \overline{L}_{\tilde{0}\tilde{\delta}}-\Gamma_{\tilde{\gamma} \tilde{\delta}}^{\tilde{\beta}} \overline{L}_{\tilde{\beta} \tilde{0}} )\\
&=O(r).
\end{align*}
 
This gives us 
 \begin{equation*}
     L^{\tilde{\gamma}\tilde{\delta}}W^{g_+}_{\tilde{\gamma} \tilde{0}\tilde{\delta} \hat{\bar{n}}}f\nu_{\varepsilon}^{\tilde{0}}=O(r^3)
 \end{equation*}
and
\begin{align*}
L^{\tilde{\gamma}\tilde{\delta}}W^{g_+}_{\tilde{\gamma}\tilde{\alpha}\tilde{\delta} \hat{\bar{n}}}f\nu_{\varepsilon}^{\tilde{\alpha}}&=L^{\tilde{\gamma}\tilde{\delta}}W_{\tilde{\gamma} \tilde{0}\tilde{\delta} \hat{\bar{n}}}f\nu_{\varepsilon}^{\tilde{0}}+L^{\tilde{\gamma}\tilde{\delta}}W^{g_+}_{\tilde{\gamma} \tilde{b}\tilde{\delta} \hat{\bar{n}}}f\nu_{\varepsilon}^{\tilde{b}}\\
&=L^{\tilde{\gamma}\tilde{\delta}}W^{g_+}_{\tilde{\gamma} \tilde{0}\tilde{\delta} \hat{\bar{n}}}f\nu_{\varepsilon}^{\tilde{0}}+O(r^3)\\
&=O(r^3).
\end{align*}
Therefore we may write
 \begin{equation*}
 L^{\tilde{\gamma}\tilde{\delta}}W^{g_+}_{\tilde{\gamma}\tilde{\alpha}\tilde{\delta} \hat{n}}f\nu_{\varepsilon}^{\tilde{\alpha}}dv_k= O(\varepsilon)dv_{\bar{k}}.
 \end{equation*}

We then get that $$\oint_{\Sigma_{\varepsilon}} L^{\tilde{\gamma}\tilde{\delta}}W^{g_+}_{\tilde{\gamma}\tilde{\alpha}\tilde{\delta} \hat{n}}\nu_{\varepsilon}^{\tilde{\alpha}}f dv_k=O(\varepsilon). $$
This proves the claim.
\end{proof}
Combining Claim \ref{lastclaim} with (\ref{deriveq}) and (\ref{Wtermeq}) and letting $\varepsilon \to 0$ yields the theorem.
\end{proof}

As promised in the introduction, we show that the finite part can be written as a
convergent integral.
\begin{lemma}
\label{finlem}
With notation as above, we obtain
\begin{equation*}
f.p. \int_{Y_{\varepsilon}} |L|^2fdv_{\tilde{h}}=\int_{Y} \left(\Delta_{\h} (|L|^2f)+|L|^2f\right) dv_{\tilde{h}},
\end{equation*}
where the right side is a convergent integral.

\end{lemma}
\begin{proof}
By \eqref{LL} we know 
\begin{equation}
|L|_{\h}^2f=|\mathring{\overline{II}}_M|^2_{\overline{k}}\tilde{f}r+O(r^3\log(r)) 
\end{equation}
which implies 
 \begin{align*}
 \nabla_{\nu_{\varepsilon}} |L|_{\h}^2f&=
 \nabla_{\nu_{\varepsilon}}\Big[|\mathring{\overline{II}}_M|^2_{\overline{k}}\tilde{f}r+O(r^3\log(r)) \Big]\\&=|\mathring{\overline{II}}_M|^2_{\overline{k}}\tilde{f}\varepsilon+O(\varepsilon^3\log(\varepsilon)).\nonumber
 \end{align*}
 It follows that 
 \begin{equation}
 \oint_{\Sigma_{\epsilon}}\nabla_{\nu_{\varepsilon}}(|L|^2_{\tilde{h}}f)  dv_{k_\varepsilon}=\varepsilon^{-1}\oint_{\Sigma}|\mathring{\overline{II}}_M|^2_{\overline{k}}\tilde{f} dv_{\overline{k}}+O(\varepsilon\log(\varepsilon)),
 \end{equation}
 where we have used that fact that $\sqrt{\mathrm{det}\hspace{1mm} \overline{k}_{\varepsilon}}$ has vanishing first derivative at $r=0.$
 By Stokes's theorem
 \begin{equation}
 \int_{Y_{\varepsilon}}\Delta^Y(|L|^2_{\tilde{h}}f) dv_{\tilde{h}}=-\oint_{\Sigma_{\epsilon}}\nabla_{\nu_{\varepsilon}}(|L|^2_{\tilde{h}}f)  dv_{\overline{k}_\varepsilon}.
 \end{equation}
 Also recall \eqref{LLL}:
 \begin{align*}
 \int_{Y_{\varepsilon}}|L|^2_{\tilde{h}}f dv_{\tilde{h}}&=C+\oint_{\Sigma} \int_{\varepsilon}^{r_0}|\mathring{\overline{II}}_M|^2_{\overline{k}}\tilde{f}r^{-2}+O(1)  drdv_{\overline{k}} \\\nonumber
 &= f.p.\int_{Y_{\varepsilon}}|L|^2fdv_{\h} + \varepsilon^{-1}\oint_{\Sigma}|\mathring{\overline{II}}_M|^2_{\overline{k}}\tilde{f}dv_{\overline{k}}+O(\varepsilon).
 \end{align*}
 The result now follows.
\end{proof}
\section{Appendix} \label{Appendix} 

In this appendix we give a brief summary of the formulas needed in the proof of Theorem \ref{varthm}, based on notes provided by Nicholas Edelen.  Although they are all standard, due to differences in notation and convention we have decided to present a summary of the calculations.   

Let $(X,g)$ be a Riemannian manifold of dimension $n+1$, and $\nabla$ denote the Riemannian connection.  Let $Y$ be a smooth manifold of dimension $n$, and consider a one-parameter family of smooth immersions $\mathcal{F} : (-\epsilon, \epsilon) \times Y \rightarrow X$.  Let $h = (\mathcal{F}_t)^{*}g$ be the induced metric on $Y$, and $\nabla^Y$ the corresponding connection.  

Let $V$ denote the variation field of $\mathcal{F}_t$: 
\begin{align*}
V = \left.\frac{d}{dt} \mathcal{F}_t \right|_{t = 0}.
\end{align*}
Eventually we will assume that $\mathcal{F}_t$ is a normal variation; i.e., if $\nu$ is a choice of unit to $Y$ then there is a function $f \in C^{\infty}(Y)$ such that $V = f \nu$.  

Let $\{ x^1, \dots, x^n \}$ be local coordinates near a point $0 \in Y$.  They induce coordinates on $\mathcal{F}_t(Y)$ defined via $(t,x^1,\dots,x^n) \mapsto \mathcal{F}_t(x^1,\dots,x^n)$, and we have the corresponding coordinate vector fields $\{ \partial_1, \dots, \partial_n \}$, along with $\partial_t = V$.  Let
\begin{align*}
h_{\alpha \beta}(t,x) = g_{\mathcal{F}_t(Y)}(\partial_{\alpha}, \partial_{\beta}).
\end{align*}
Then
\begin{align*}
h'_{\alpha \beta} &= \left.\frac{\partial}{\partial t} h_{\alpha \beta} \right|_{t = 0} \\
&= g(\nabla_{\partial_t} \partial_{\alpha},  \partial_{\beta} ) + g( \partial_{\alpha}, \nabla_{\partial_t} \partial_{\beta}) \\
&= g(\nabla_{\partial_{\alpha}} V, \partial_{\beta}) + g( \partial_{\alpha}, \nabla_{\partial_{\beta}} V).
\end{align*}
If $V = f \nu$, then this becomes 
\begin{align} \label{hdot1}
h'_{\alpha \beta} = f g(\nabla_{\partial_{\alpha}} \nu, \partial_{\beta}) + g( \partial_{\alpha}, \nabla_{\partial_{\beta}} \nu).
\end{align}
Given a choice of normal $\nu$ our definition of the second fundamental form of $Y$ is
\begin{align} 
L(\partial_{\alpha}, \partial_{\beta}) = g(\nu, \nabla_{\partial_{\alpha}} \partial_{\beta}) = - g(\nabla_{\partial_{\alpha}} \nu, \partial_{\beta}).
\end{align}
Therefore, by (\ref{hdot1}) we conclude 
\begin{align} \label{hdot}
h'_{\alpha \beta} = - 2 f L_{\alpha \beta}.
\end{align}
By the standard formula for the inverse, this implies 
\begin{align} \label{hinvdot}
	(h^{\alpha \beta})' = 2 f L^{\alpha}{}_{\gamma} L^{\beta \gamma}.
\end{align}

By our definition of second fundamental form, 
\begin{align} \label{L1} \begin{split}
L'_{\alpha \beta} &= \left.\frac{\partial}{\partial t} L_{\alpha \beta} \right|_{t = 0} \\
&= g(\nabla_{\partial_t} \nu, \nabla_{\partial_{\alpha}} \partial_{\beta} ) + g( \nu, \nabla_{\partial_t} \nabla_{\partial_{\alpha}} \partial_{\beta} ).
\end{split}
\end{align}
The first term on the right is easily seen to vanish, since $0 = \partial_t g(\nu,\nu) = 2 g( \nabla_{\partial_t} \nu, \nu )$ implies that 
\begin{align}  \label{p1}
g(\nabla_{\partial_t} \nu, \nabla_{\partial_{\alpha}} \partial_{\beta} ) = - L_{\alpha \beta} \, g( \nabla_{\partial_t} \nu, \nu ) = 0.
\end{align}
For the second term, we commute derivatives to get 
\begin{align}  \label{p2} \begin{split}
g( \nu, \nabla_{\partial_t} \nabla_{\partial_{\alpha}} \partial_{\beta} ) &= g( \nu, \nabla_{\partial_{\alpha}} \nabla_{\partial_t}  \partial_{\beta} ) + R(V, \partial_{\alpha}, \partial_{\beta}, \nu) \\
&= g( \nu, \nabla_{\partial_{\alpha}} \nabla_{\partial_{\beta}} V ) + R(V, \partial_{\alpha}, \partial_{\beta}, \nu),
\end{split}
\end{align}
where $R$ is the curvature tensor of $g$.  If $V = f \nu$ then by (\ref{p1}) and (\ref{p2}), (\ref{L1}) simplifies to 
\begin{align} \label{plug} \begin{split}
L'_{\alpha \beta} &= g( \nu, \nabla_{\partial_t} \nabla_{\partial_{\alpha}} \partial_{\beta} ) \\
&= g( \nu, \nabla_{\partial_{\alpha}} \nabla_{\partial_{\beta}} (f \nu) ) + f R(\nu, \partial_{\alpha}, \partial_{\beta}, \nu) \\
&= \nabla^Y_{\alpha} \nabla_{\beta}^Y f  + g(\nu, \partial_{\alpha} f \nabla_{\partial_{\beta}} \nu +  \partial_{\beta} f \nabla_{\partial_{\alpha}} \nu + f \nabla_{\partial_{\alpha}} \nabla_{\partial_{\beta}} \nu ) + f R(\nu, \partial_{\alpha}, \partial_{\beta}, \nu) \\
&= \nabla_{\alpha}^Y \nabla_{\beta}^Y f + f g(\nu, \nabla_{\partial_{\alpha}} \nabla_{\partial_{\beta}} \nu ) + f R(\nu, \partial_{\alpha}, \partial_{\beta}, \nu),
\end{split}
\end{align}
where in the last line we used the fact that $\partial_{\alpha} g(\nu,\nu) = 0$.  Using this fact again we also find
\begin{align} \label{plug2}
g( \nu, \nabla_{\partial_{\alpha}} \nabla_{\partial_{\beta}} \nu ) &= - g( \nabla_{\partial_{\alpha}} \nu, \nabla_{\partial_{\beta}} \nu).
\end{align}
It follows from the definition of the second fundamental form that 
\begin{align*}
\nabla_{\partial_{\alpha}}\nu = - L_{\alpha}^{\gamma} \, \partial_{\gamma},
\end{align*}
hence
\begin{align*}
- g( \nabla_{\partial_{\alpha}} \nu, \nabla_{\partial_{\beta}} \nu) = - L_{\alpha}^{\gamma} L_{\beta \gamma}.
\end{align*}
Substituting this into (\ref{plug2}) and combining with (\ref{plug}), we arrive at  
\begin{align} \label{Ldot}  
L'_{\alpha \beta} = \nabla_{\alpha}^Y \nabla_{\beta}^Y f  - f \, L_{\alpha}^{\gamma} L_{\beta \gamma} + f \, R(\nu, \partial_{\alpha}, \partial_{\beta}, \nu).
\end{align}

For the variation of the mean curvature $H = h^{\alpha \beta} L_{\alpha \beta}$ we use (\ref{hinvdot}) and (\ref{Ldot}) to obtain 
\begin{align} \label{Hdot}
H' = \Delta_Y f + \big( |L|^2 + \ric(\nu,\nu) \big)f. 
\end{align}
Finally, using the standard formula for the derivative of the volume form, we have 
\begin{align} \label{dVdot}
(dv_h)' = - f H \, dv_h. 
\end{align}

\bibliographystyle{alpha}
\bibliography{gbrv}
\end{document}